\documentclass[a4paper, 10pt, twoside]{amsart}
\usepackage[top=3cm, bottom=3cm, left=3cm, right=3cm]{geometry}
\usepackage[pdftex]{graphicx}

\usepackage[utf8]{inputenc}
\usepackage[T1]{fontenc}

\usepackage[english]{babel}
\usepackage[noadjust]{cite}

\usepackage{amsthm,amssymb,epsfig,mathtools}
\usepackage{mathrsfs}  
\usepackage{bbm}        
\usepackage{esint}     
\usepackage{dsfont}
\usepackage{color}
\usepackage{enumerate}
\usepackage{url}
\usepackage{algorithm2e}

\usepackage{calc}
\usepackage[inline]{enumitem}
\usepackage[normalem]{ulem}
\usepackage{url}

\usepackage[bookmarks, bookmarksnumbered, pdfstartview={XYZ null null 1.00}]{hyperref}

\usepackage{tikz}
\usetikzlibrary{shapes, arrows, arrows.meta, decorations.markings}

\usepackage{todonotes}
\usepackage[ulem=normalem,draft]{changes}
\usepackage{cancel}
\usepackage{comment}

\makeindex

\numberwithin{equation}{section}

\newcommand{\theoname}{Theorem}
\newcommand{\lemmname}{Lemma}
\newcommand{\coroname}{Corollary}
\newcommand{\propname}{Proposition}
\newcommand{\definame}{Definition}

\newcommand{\remkname}{Remark}
\newcommand{\explname}{Example}

\theoremstyle{plain}
\newtheorem{theorem}{\theoname}[section]
\newtheorem{lemma}[theorem]{\lemmname}

\newtheorem{proposition}[theorem]{\propname}

\theoremstyle{definition}
\newtheorem{definition}[theorem]{\definame}
\newtheorem{remark}[theorem]{\remkname}
\newtheorem{example}[theorem]{\explname}

\newlist{hypothesis}{enumerate}{1}
\setlist[hypothesis]{label={\textup{(H\arabic*)}}, ref={(H\arabic*)}, leftmargin=*, widest*=10}

\newenvironment{block}%
  {\list{}{\leftmargin=.5in\rightmargin=.5in}  \item[]  }%
  {\endlist}

\def\dd{{\,\rm d}}
\newcommand{\eqdef}{\ensuremath{\stackrel{\mbox{\upshape\tiny def.}}{=}}}

\newcommand{\norm}[1]{\left\lVert#1\right\rVert}
\newcommand{\inner}[1]{\left\langle#1\right\rangle}

\def\1B{{\bf  1}}

\def\dist{\mathop{\rm dist}}

\def\dom{\mathop{{\rm dom}}}

\def\supp{\mathop{\rm supp}}

\newcommand{\cvweak}[1]{\xrightharpoonup[#1]{}}
\newcommand{\cvstrong}[2]{\xrightarrow[#1]{#2}}

\DeclareMathOperator*{\minimize}{minimize}

\def\inf{\mathop{\rm inf}}
\def\sup{\mathop{\rm sup}}

\def\min{\mathop{\rm min}}
\def\max{\mathop{\rm max}}

\def\argmin{\mathop{\rm argmin}}

\newcommand{\mres}{\mathbin{\vrule height 1.6ex depth 0pt width
		0.13ex\vrule height 0.13ex depth 0pt width 1.3ex}}

\newcommand{\X}{\mathcal{X}}
\newcommand{\Y}{\mathcal{Y}}
\newcommand{\gammab}{\boldsymbol{\gamma}}
\newcommand{\nub}{\boldsymbol{\nu}}
\newcommand{\mub}{\boldsymbol{\mu}}
\DeclareMathOperator{\capp}{\mathrm{}{cap}}

\title{From Nash to Cournot--Nash equilibria via \texorpdfstring{$\Gamma$}{Gamma}-convergence}

\author{João Miguel Machado}
\address{Lagrange Mathematical and Computational Center\\
103 rue de Grenelle\\
Paris, 75007}
\email{joao-miguel.machado@ceremade.dauphine.fr}

\author{Guilherme Mazanti}
\address{Université Paris-Saclay, CNRS, CentraleSupélec, Inria, Laboratoire des signaux et systèmes \& Fédération de Mathématiques de CentraleSupélec, 91190, Gif-sur-Yvette, France.}
\email{guilherme.mazanti@inria.fr}

\author{Laurent Pfeiffer}
\address{Université Paris-Saclay, CNRS, CentraleSupélec, Inria, Laboratoire des signaux et systèmes \& Fédération de Mathématiques de CentraleSupélec, 91190, Gif-sur-Yvette, France.}
\email{laurent.pfeiffer@inria.fr}

\begin{document}

\begin{abstract}
This work addresses the issue of the convergence of an $N$-player game towards a limit model involving a continuum of players, as the number of agents $N$ goes to infinity. More precisely, we investigate the convergence of Nash equilibria to equilibria of Cournot--Nash type. When the cost function of the players is the first variation of some potential function, equilibria can be characterized by a stationarity condition, satisfied in particular by the minimizers of the potential. This characterization is shown under low regularity assumptions. We give a positive answer to the convergence question when players interact in a pairwise fashion; in this case we show that the original sequence of $N$-player games also admit a potential structure and prove that their corresponding potential functions converge in the sense of $\Gamma$-convergence to the potential function of the limit game. Aligned with the strong characterization of equilibria, we can show that any convergent sequence of Nash equilibria has a Cournot--Nash equilibrium as limit. 

\bigskip


\medskip

\end{abstract}

\maketitle

\setcounter{tocdepth}{1}
\tableofcontents

\section{Introduction}\label{sec:introduction}
This work deals with a class of $N$-player games. Our main objective is to investigate the convergence of Nash equilibria, as $N$ goes to infinity, to Cournot--Nash equilibria of a game with a continuum of players (see Definition~\ref{def:equ_CournotNash}).

Let us first describe the class of $N$-player games of interest in this paper. Let $\X$ and $\Y$ be Polish spaces, that is, separable and completely metrizable topological spaces, representing the space of types of players and of admissible strategies, respectively. Given a positive integer $N$ and a tuple ${\left(x_i\right)}_{i = 1}^N \subset \X$, where $x_i$ corresponds to the type of player $i$, we assume that the goal of each player is to minimize the following cost:
\begin{equation}\label{eq.toy_game}
    \minimize_{y_i \in \Y}
    g_N(y_i;y_{-i}) 
    \eqdef 
    c(x_i,y_i) + L(y_i)
    + 
    \frac{2}{N}\sum_{j \neq i}H(y_j,y_i),
\end{equation}
where, as standard in game theory, we use $y_{-i}$ to denote the tuple $(y_j)_{j\neq i}$. We use the terminology \emph{pairwise} to describe the fact that $g_N$ involves the average value of the quantity $H(y_j,y_i)$, modelling the specific interaction between the pair of  players $i$ and $j$. Note that the functions $c$, $L$, and $H$ are assumed to be the same for all players, the difference between players being represented only by their types $x_1, \dotsc, x_N$. A profile of strategies ${\left(y_i\right)}_{i = 1}^N$ satisfying the minimization conditions from~\eqref{eq.toy_game} simultaneously for all $i = 1,\dots, N$ is called a Nash equilibrium in pure strategies. The formulation in mixed strategies consists of a relaxation of this game in which, for each player $i$, we replace the set of strategies $\Y$ for the space of Radon of measures $\mathscr{P}(\Y)$ and the minimization becomes 
\begin{equation}\label{eq.toy_game_mixed}
    \minimize_{\nu_i \in \mathscr{P}(\Y)}
    g_N(\nu_i;\nu_{-i}) 
    \eqdef 
    \int_{\Y}\left(c(x_i,y_i) + L(y_i)\right)\dd\nu_i(y_i)
    + 
    \frac{2}{N}\sum_{j \neq i}
    \int_{\Y\times\Y}
    H\dd \nu_j\otimes\nu_i.
\end{equation}

Such types of $N$-player games become rapidly intractable as $N$ grows and, for this reason, one may wonder whether a suitable notion of equilibrium in a setting with infinitely many players will provide information for the behavior of equilibria when $N$ is large but still finite. This question is further motivated by the fact that, in many economic and social scenarios, the number of agents acting in the game reaches this regime of intractability, and hence it is of great theoretical and practical importance to be able to rigorously describe the model with infinitely many players as the limit of a sequence of $N$-player games. In the present work we consider a fixed distribution of types of players $\mu \in \mathscr{P}(\X)$. We will call Cournot--Nash equilibrium (see Definition~\ref{def:equ_CournotNash}) a measure $\gamma \in \mathscr{P}(\X\times\Y)$ such that ${(\pi_\X)}_\sharp \gamma = \mu$ and such that
\begin{equation}\label{eq.toy_gameCournotNash}
    y \in \argmin_{y' \in \Y} \Bigg\{ c(x,y') + L(y') + 2\int_\Y H(y',\bar y) \dd \nu(\bar y) \Bigg\}, \quad
    \text{ for $\gamma$-a.e. } (x,y)\in \X\times\Y,
\end{equation}
where $\nu = {(\pi_\Y)}_\sharp \gamma$.

In the category of differential games, the notion of equilibrium in a continuum-player setting was independently introduced around the same time by Caines, Huang, and Malhamé in \cite{Huang2003Individual, Huang2007Large, Huang2006Large} and by Lasry and Lions in \cite{Lasry2006JeuxI, Lasry2006JeuxII, Lasry2007Mean}, the \emph{Mean Field Games} (MFGs). In this formulation, an equilibrium is characterized by a pair of coupled PDEs, one describing the evolution of the distribution of players and another encoding the optimality conditions of each player’s control problem. The convergence of $N$-player games to MFGs was a primary motivation since the inception of the theory, as discussed in Lions’ lectures at Coll\`ege de France \cite{Lions2006Jeux} and the notes by Cardaliaguet \cite{cardaliaguet2010notes}. Several convergence results are now available, mainly for second-order MFGs \cite{Kolokoltsov2014Rate, cardaliaguet2017convergence, Lauriere2022Convergence, Lacker2021Mean, Cardaliaguet2019Master}, though many questions remain open, particularly for first-order MFGs.

More generally, the concept of games with a continuum of players dates back to the 1960s in the economics literature, with seminal contributions by Aumann \cite{aumann1964markets,aumann1966existence,mas1984theorem}. Schmeidler \cite{schmeidler1973equilibrium} studied such models with the goal of proving the existence of pure-strategy equilibria.
Subsequent works~\cite{hart1974equilibrium,mas1984theorem} relaxed this notion to allow probability measures over strategies, giving rise to what is now known as a \emph{Cournot--Nash equilibrium} (see Definition~\ref{def:equ_CournotNash}).

A similar relaxation also played an important part in the development of the \textit{optimal transportation problem} (OT), described as follows: given two Polish spaces $(\X, d_\X)$ and $(\Y, d_\Y)$, a pair of Borel probability measures $(\mu, \nu) \in \mathscr{P}(\X) \times \mathscr{P}(\Y)$, and a transportation cost $c\colon\X\times \Y \to \mathbb{R}\cup\{+\infty\}$, one seeks to solve the minimization problem
\begin{equation*}
    \mathcal{W}_c(\mu,\nu)
    \eqdef 
    \min_{\gamma \in \Pi(\mu, \nu)} \int_{\X \times \Y} c(x,y) \dd \gamma.
\end{equation*}
The minimum on the right side is called Kantorovitch formulation~\cite{kantorovich1942translocation}, which is taken among all couplings $\gamma$ of $\mu$ and $\nu$, \emph{i.e.}, $\gamma$ is chosen in
\begin{equation}
    \Pi(\mu, \nu) 
    \eqdef
    \left\{
        \gamma \in \mathscr{P}(\X\times \Y): 
        (\pi_\X)_\sharp \gamma = \mu, \  (\pi_\Y)_\sharp \gamma = \nu      
    \right\},
\end{equation}
the space of probability measures in the product space $\X \times \Y$ whose marginals are $\mu$ and $\nu$. 

In the context of the games \eqref{eq.toy_game} and \eqref{eq.toy_game_mixed}, recall that $\X$ is the space of types of players and $\Y$ is the space of admissible strategies for the players. Given a distribution of types of players $\mu \in \mathscr{P}(\X)$ and a distribution of strategies $\nu \in \mathscr{P}(\Y)$, a coupling $\gamma \in \Pi(\mu,\nu)$ represents the joint distribution of players and strategies, \emph{i.e.}, given $A \times B \subset \X \times \Y$, the quantity $\gamma(A\times B)$ represents the probability that a player has type in $A$ and chooses a strategy in $B$.

To introduce the notion of Cournot--Nash equilibrium, we consider a more general model than \eqref{eq.toy_gameCournotNash}, in which a player of type $x \in \X$ wishes to choose a strategy $y \in \Y$ in order to minimize the cost $\Phi(x, y, \nu)$, where $\nu \in \mathscr P(\Y)$ describes the mean field of strategies of the agents and $\Phi \colon \X \times \Y \times \mathscr{P}(\Y) \to \mathbb{R}\cup \{+\infty\}$ is a given cost function. 
 
\begin{definition}\label{def:equ_CournotNash}
Let $\mu \in \mathscr P(\X)$ and $\Phi\colon\X\times\Y\times\mathscr{P}(\Y) \to \mathbb R \cup \{+\infty\}$. We say that a probability measure $\gamma \in \mathscr P(\X \times \Y)$ is a \emph{Cournot--Nash equilibrium} for the game with cost $\Phi$ and with distribution of types $\mu$ if $(\pi_{\X})_{\sharp} \gamma = \mu$ and $\gamma$ satisfies the \emph{equilibrium condition} 
\begin{equation}\label{eq:CN}
\gamma 
\left(
\left\{
(x,y) \in \X\times \Y : 
y \in \argmin_{y' \in \Y} 
\Phi(x,y',\nu)
\right\}
\right) = 1,
\end{equation}
where $\nu = (\pi_{\Y})_{\sharp} \gamma$. It is called an \emph{equilibrium of finite social cost} if 
\begin{equation}\label{eq:finite_social_cost_condition}
        \int_{\X\times\Y}\Phi(x,y,\nu)\dd\gamma < +\infty.
\end{equation}
\end{definition}

Results guaranteeing the existence of equilibria have been established with fixed point methods in the above-mentioned works~\cite{aumann1964markets,aumann1966existence,mas1984theorem}. 
This approach strongly relies on the continuity of the cost function. In~\cite{blanchet2016optimal}, whenever $\Phi$ has a potential structure, \emph{i.e.}, it can be written as 
\begin{equation}\label{eq:cost_potential_structure}
    \Phi(x,y,\nu) 
    = 
    c(x,y) 
    + 
    \frac{\delta \mathcal{E}}{\delta \nu}[\nu](y),
\end{equation}
the sum of an individual continuous cost $c(x,y)$ and the first variation of a functional $\mathcal{E}$ (see Definition~\ref{def.first_var_criticalpoint}), Blanchet and Carlier showed that, if
\begin{equation}\label{eq:variational_characterization_BlanchetCarlier}
    \nu \in \argmin_{\nu' \in \mathscr{P}(\Y)} 
    \mathcal{W}_c(\mu,\nu') + \mathcal{E}(\nu')
\end{equation}
and $\gamma \in \Pi(\mu,\nu)$ is an optimal transportation plan for the cost $c$, then $\gamma$ is a Cournot--Nash equilibrium in the sense of Definition~\ref{def:equ_CournotNash}. As their proof of existence is of variational nature, it provides a natural approach to compute equilibria numerically, as done in~\cite{blanchet2014remarks,blanchet2016optimal} (see also~\cite{blanchet2018computation} for an approach using entropic regularization of the optimal transportation term $\mathcal{W}_c(\mu,\nu')$). 


\subsection{Contributions of this work}

Our main goal is to answer the following question:
\begin{block}
    \em
    Given a sample of players following a continuous distribution, when will a sequence of Nash equilibria for the associated finite game converge to a Cournot--Nash equilibrium?
\end{block}
More precisely, we fix a reference probability space $(\Omega, \mathcal{F}, \mathbb{P})$ and a sequence ${\left(X_i\right)}_{i \in \mathbb{N}}$ of independent and identically distributed (i.i.d.)\ $\mathcal X$-valued random variables, adap\-ted to this reference space and having common law $\mu \in \mathscr{P}(\X)$. This sequence represents a sample of players, where the first $N$ elements describe the type of the agents in our $N$-player game. 

Our sequence of $N$-player games will then resemble~\eqref{eq.toy_game}, where the points ${\left(x_i\right)}_{i = 1}^N$ represent a realization of this sample of players. We will answer this question for two types of \emph{information structure}, in \textit{closed loop} and in \textit{open loop information structure}. The former corresponds to the situation where players choose their strategies having knowledge of the realization of the sample. In the latter, the decision is taken without the knowledge of this realization, so that strategy profiles assume the form of an optimal execution strategy with respect to the event $\omega$ of each realization, see Sections~\ref{sec:conv_closed_loop} and~\ref{sec:conv_open_loop} for more details on each formulation. 

To answer this question, we first go back the the variational structure from~\cite{blanchet2016optimal} discussed above. We give a complete characterization of Cournot–Nash equilibria in terms of the stationarity of a potential functional (Theorem~\ref{thm:potential_structure_cournot_nash}). As a result, we extend the framework of Blanchet–Carlier, which was formulated in terms of minimality and relied heavily on regularity of the cost functions, whereas our arguments require only minimal assumptions. We assume that $c$ and $\mathcal{E}$, as in~\eqref{eq:cost_potential_structure}, are lower semi-continuous and $\mathcal{E}$ admits a first variation which has compact level sets. Instead of working with the energy $\nu \mapsto \mathcal{W}_c(\mu,\nu) + \mathcal{E}(\nu)$, we use a lifted energy over the space of transportation plans with fixed marginal $\mu$, which is defined as
\begin{equation}\label{eq:J_intro}
    \mathcal{J}(\gamma) 
    \eqdef 
    \begin{dcases}
        \int_{\X\times \Y} c(x, y) \dd \gamma + \mathcal{E}\left(\nu\right),
        & \text{ if } \gamma \in \mathscr{P}_\mu(\X\times \Y),\ {(\pi_{\Y})}_{\sharp} \gamma = \nu,\\ 
        +\infty,
        & \text{ otherwise},
    \end{dcases}
\end{equation}
where 
\begin{equation}
    \mathscr{P}_\mu(\X\times \Y)
    \eqdef 
    \left\{
        \gamma \in \mathscr{P}(\X\times\Y) : 
        {\left(\pi_\X\right)}_\sharp\gamma = \mu
    \right\}. 
\end{equation}
We show that $\gamma$ is a Cournot--Nash equilibrium if, and only if, it is a critical point of the energy $\mathcal{J}$ (see Definition~\ref{def.first_var_criticalpoint}). 

In the sequel, we focus our attention to the case where $\mathcal{E}$ can be decomposed into a mean individual and interaction energies, \emph{i.e.},
\begin{equation}\label{eq:energies_pairwise}
    \mathcal{E}(\nu) 
    = 
    \mathcal{L}(\nu) + \mathcal{H}(\nu,\nu), 
    \text{ where }
    \mathcal{L}(\nu) \eqdef \int_\Y L\dd\nu \text{ and } 
    \mathcal{H}(\nu,\bar\nu) \eqdef \int_{\Y\times\Y} H\dd \nu\otimes\bar\nu.
\end{equation} 
This is the framework for which we shall answer the convergence question, see Hypothesis~\ref{Hypo-atomless}--\ref{Hypo-compact} for the exact technical assumptions. With the supplementary assumption~\ref{Hypo-gluing_operator}, we also establish a Lipschitz stability property of the minimal value of the potential with respect to the distribution of player types (Theorem~\ref{thm:stability_value_function}). 
\[
    \left|
        \inf_{\gamma \in \mathscr{P}_{\mu_0}(\X\times \Y)} \mathcal{J} 
        - 
        \inf_{\gamma \in \mathscr{P}_{\mu_1}(\X\times \Y)} \mathcal{J}
    \right|
    \le C W_1(\mu_0, \mu_1),
\]
where $W_1$ corresponds to the $1$-Wasserstein distance, see the definition in Section~\ref{sec:stability-value-function} below. This ensures robustness of equilibria under perturbations of the population.

With this new characterization of equilibria via stationarity, we come back to the question of convergence of Nash to Cournot--Nash equilibria. We show that both families of $N$-player games, in open and closed loop information structure, admit potential functionals, a functional whose minimization yields Nash equilibria, respectively $\mathcal{J}_{\Omega,N}$ and $\mathcal{J}_{\omega,N}$ defined in~\eqref{eq:potential_function_openloop} and~\eqref{eq:potential_function_closedloop} below. Next we show that both potentials $\Gamma$-converge to the continuum potential $\mathcal{J}$. In the open-loop setting, this convergence is obtained in the narrow topology of random measures, whereas in the closed-loop setting, we achieve $\Gamma$-convergence with full probability (Theorems~\ref{theorem:gamma_conv_openloop} and~\ref{theorem:closed_loop_gamma_conv}). To our knowledge, results combining convergence in full-probability with hard marginal constraints are scarce in the literature, highlighting a major difficulty in our analysis. By the fundamental property of $\Gamma$-convergence, all accumulation points of sequences of minimizers are minimizers of the limiting functional (see Section~\ref{subsec:gamma_convergence}). Since minimizers of $\mathcal{J}_{\Omega,N}$ and $\mathcal{J}_{\omega,N}$ correspond to Nash equilibria, it follows that their accumulation points are Cournot--Nash equilibria in the continuum game. 

We obtain a full answer to our question by combining the $\Gamma$-convergence results with the stronger stationarity characterization introduced earlier. Indeed, taking the limit as $N \to \infty$ of a sequence of equilibria of the $N$-player games, we obtain the stationarity condition of $\mathcal{J}$ for the limit measure, showing that limits of any sequence of Nash equilibria are Cournot--Nash equilibria, under the slightly stronger assumption that $H$ is continuous and bounded. Convergence of minimizers of the potential functions is very natural from a variational perspective, but from a game-theory point-of-view it implies the convergence of very specific profiles of equilibria, where agents collaborate to minimize the potential function, which can be interpreted as a global utility. Combining the variational convergence with the stationarity-based characterization of equilibria yields a much stronger convergence result for finite-player games. 

This result is surprising in the sense that we obtain the convergence of non-variational objects, all Nash equilibria, even the ones that are not obtained via a variational principle, by means of a variational argument. Another advantage of employing this perspective is that, given a Cournot--Nash equilibrium $\gamma$ of the game with a continuum of players, the recovery sequence associated with $\gamma$ obtained from the $\Gamma$-convergence provides a sequence of natural candidates for being $\varepsilon$-Nash equilibria of the corresponding $N$-player games
\color{black}

\subsection{Related work}

In this paragraph we describe previous results from the literature that are similar to ours thematically or in terms of the tools employed. 

\subsubsection*{Convergence of Nash equilibria} 
As mentioned before, the convergence of Nash equilibria to equilibria in a continuous setting is a reoccurring idea in the Game Theory literature. Next we give a non-exhaustive list of advancements in this direction:
\begin{itemize}
    \item \cite{Cardaliaguet2019Master} and~\cite{carlier2017convergence} study the convergence question in the framework of MFGs by means of the Master equation, an equation in the space of probability measures that gives global information on the optimality of the MFG system;
    \item the study of convergence of Nash to Cournot--Nash equilibria was conducted in~\cite{blanchet2014nash}, under the assumption that the sequence of functions defining the problem 
    $\Phi_N \colon \X \times \Y \times \mathscr{P}(\Y) \to \mathbb R$, as in Definition~\ref{def:equ_CournotNash}, are uniformy Lipschitz with respect to $(d_\X,d_\Y,\mathcal{W}_1)$ and converge uniformly to $\Phi$. While our problem enjoys the pairwise (and also potential) structure, our regularity assumptions are much weaker and allow the treatment of many examples, \emph{cf.}\ Subsection~\ref{subsection:examples}. 
\end{itemize}

\subsubsection*{\texorpdfstring{$\Gamma$}{Gamma}-convergence particle systems}

Our methods are conceptually close to those of Serfaty in~\cite{serfaty2015coulomb} for the study of limits of particle systems under Coulomb-type interactions. The major difference lies in the the probabilistic nature of our convergence results and the further complexity of the first marginal constraint, natural to our problem to fix the distribution of players $\mu$ (or $\mu_N$ in the $N$-player case).

On the probabilistic side of our $\Gamma$-convergence results,~\cite{garcia2016continuum} deals with the continuum limit of the \text{total variation functional} defined on graphs induced by i.i.d.\ point clouds. The authors prove that a mollified version of a total variation functional, w.r.t.~an i.i.d.\ sample ${(X_i)}_{i \in \mathbb{N}}$ of a continuous law, $\Gamma$-converges with full probability to the weighted total variation functional. Their results are conditional to a good scaling of the mollification parameter, being of the order of the optimal quantization problem, which can be precisely estimated as $N \to \infty$ with full probability. This is the only probabilistic element in their analysis, and afterwards their methods become purely deterministic.

\subsubsection*{Convergence results for averaged optimal control problems}
Related results can also be found in the literature of \textit{stochastic optimization} since the seminal work~\cite{artstein1995consistency}, where convergence of minimizers is obtained via the notion of epi-convergence of a sequence of empirical averages of a functional, commonly written as a expectation. The relation of our methods becomes clear when one notices that the definition of epi-convergence is equivalent to Kuratowski convergence of the epigraph of the sequence of functionals that $\Gamma$-converge, see for instance~\cite{rockafellar1998variational}. These methods have become relevant in the community of \textit{averaged optimal control}, see for instance~\cite{vinter2005minmax,phelps2016optimal} where the consistency of a numerical method based on empirical approximations is proven in this context, and  more recently~\cite{aronna2025average}.

\subsection{Examples}\label{subsection:examples}
In this paragraph we discuss multiple examples that are covered by our model and their relevance in the literature. 

\subsubsection*{Potential Cournot--Nash equilibria \cite{blanchet2016optimal}}
The first clear example is the model proposed by Blan\-chet and Carlier in~\cite{blanchet2016optimal}. As discussed above, they proposed a variational principle to find equilibria as in Definition~\ref{def:equ_CournotNash} and gave plenty of examples of economic applications for this model, such as the holiday choice and technology choice models. The distinctions from ours is that, in~\cite{blanchet2016optimal}, $c$ has to be a continuous cost in order to give the variational characterization of equilibria using the OT problem as in equation~\eqref{eq:variational_characterization_BlanchetCarlier}. We do not need this assumption since we propose the lifting to the space of transportation plans $\gamma \in \mathscr{P}_\mu(\X\times\Y)$. This lift is purely technical and the characterization via the value of the associated OT problem still holds in our case and is useful for numerical purposes, since one can use the dual formulation of the OT problem as a dimensionality-reduction technique. On the other hand, taking $c$ to be lower semi-continuous allows us to make a link with the next class of examples. 

\subsubsection*{Abstract Lagrangian Mean Field Games \cite{Santambrogio2021Cucker}}
Consider a crowd motion, where the starting point of each agent is distributed by a probability measure $\mu$ and the final goal of each agent is to reach a target set while minimizing a cost depending on their own trajectory and on the distribution of trajectories of all agents $Q$. One can think of the target set as the exit of a metro, for instance. In~\cite{Santambrogio2021Cucker}, Santambrogio and Shim propose a model where each agent chooses their trajectory among all possible continuous curves respecting their given initial condition. In this case, $\X = \Omega$ is a compact subset of $\mathbb{R}^d$ and $\Y = C^0([0,T];\Omega)$. Each agent then tries to find a curve $\sigma$ such that $\sigma(0) = x_0$, the given initial condition, while minimizing an energy with end-point penalties and an interaction of Cucker--Smale type giving rise to a consensus of the velocities as in the seminal paper~\cite{Cucker2007Smale}. As the measure $Q$ corresponds to the distribution of trajectories of all agents, the initial condition is then imposed by the constraint $(e_0)_\sharp Q = \mu$. They defined equilibria as measures $Q \in \mathscr{P}(\Y)$ such that $(e_0)_\sharp Q = \mu$ which can be written in Cournot--Nash form, as in Definition~\ref{def:equ_CournotNash}. In~\cite{MazantiVariational}, this model was generalized into an abstract model with pair-wise interaction and prescribed initial condition. Their notion of equilibrium is the same as in~\cite{Santambrogio2021Cucker}, but they show that equilibria are critical points of a potential functional.

This suggests a link with the previous model of Cournot--Nash equilibria and indeed, for $c(x,y) = \chi_{\pi^{-1}(y)}(x)$ we can rewrite the constraints as 
\[
    \pi_\sharp\nu = \mu \iff \mathcal{W}_c(\mu,\nu) < \infty, \text{ since }
    \mathcal{W}_c(\mu,\nu) =
    \begin{dcases}
        0,& \text{ if $\pi_\sharp \nu = \mu$},\\
        +\infty,& \text{ otherwise.}
    \end{dcases}
\]
Conversely, if we propose the lifted energy to the space of transportation plans~\eqref{eq:J_intro}, the variational criterion for Cournot--Nash equilibria from Blanchet and Carlier is of the same form as the one for Lagrangian MFGs.

\subsection{Organization of this paper}
In Section~\ref{sec:preliminaries}, we review the major tools of measure theory and $\Gamma$-convergence that are used throughout the present manuscript. In Section~\ref{sec:potential_CounotNash}, we prove the full characterization of Cournot--Nash equilibria in our potential setting and pass to stability question of the value function of our problem w.r.t.~the distribution of players, which is of independent interest. Finally, in Section~\ref{sec:convergence}, we prove our main $\Gamma$-convergence results for our sequences of open- and closed-loop $N$-player games. In Section~\ref{sec.conclusion}, we give our concluding remarks and possible future directions. 

\section{Preliminaries in measure theory and \texorpdfstring{$\Gamma$}{Gamma}-convergence}\label{sec:preliminaries}
In this section we review the mathematical tools we require and fix the notation used throughout the text. We start with the topologies of spaces of probability measures, before introducing the space of random probability measures and the compactness properties \emph{\`a la} Prokhorov that it enjoys. We finish with a brief discussion about $\Gamma$-convergence, which is the main ingredient of our proofs.

\subsection{Topologies on spaces of Radon measures}
In this paper, we shall work with general Polish spaces $(\X,d_\X)$, that is, a complete separable space equipped with a metric topology. We let $\mathscr{M}_b(\X)$ denote the space of finite Radon measures over $\X$, that is, the space of all Borel measures $\mu$ such that $\mu(K)<\infty$ whenever $K$ is compact. We let $\mathscr{P}(\X)$ denote the subspace of probability measures over $\X$, \emph{i.e.}, positive measures with unitary total mass. 

It follows from Riesz' representation theorem that $\mathscr{M}_b(\X)$ is the topological dual space of $\mathscr{C}_0(\X)$, the continuous functions $\phi$ converging to $0$ at infinity (in the sense that, given $\varepsilon > 0$, there exists a compact subset $K$ of $\X$ such that $\lvert \phi(x) \rvert < \varepsilon$ for all $x \in \X \setminus K$), see~\cite[Chapter~1]{fonseca2007modern}. This defines a norm in $\mathscr{M}_b(\X)$, the total variation norm, and also its weak-$\star$ topology. We can define other notions of weak topology by changing the set of test functions. 

\begin{definition}
    We say a sequence of measures ${\left(\mu_n\right)}_{n \in \mathbb{N}}$ \emph{converges in the narrow topology} to $\mu$ if 
    \[
        \int_\X \phi \dd \mu_n 
        \xrightarrow[n \to \infty]{}
        \int_\X \phi \dd \mu_n
        \quad \text{ for every $\phi \in \mathscr{C}_b(\X)$},
    \]
    where $\mathscr{C}_b(\X)$, is the set of continuous and bounded functions, and we write $\mu_n \xrightharpoonup[n \to \infty]{} \mu$.
\end{definition}

Notice that if $\mu_n$ converges to $\mu$ in the weak-$\star$ topology, there is no guarantee that $\mu(\X) = \lim_{n \to \infty} \mu_n(\X)$, the norm $|\cdot|(\X)$ is only l.s.c.\ for this notion of convergence. For the narrow convergence, however, as it is in duality with $\mathscr{C}_b(\X)$, we can consider the constant $1$ as test function and obtain the convergence of the total masses. This implies that $\mathscr{P}(\X)$ is closed for the narrow topology, but not for the weak-$\star$, unless $\X$ is compact. Besides this, we also have a nice criterion of compactness for the narrow topology.

\begin{theorem}[{Prokhorov's theorem, \cite[Theorem~2.8]{ambrosio2021lectures}}]\label{thm:Prokhorov}
    Let $\mathcal{F} \subset \mathscr{P}(\X)$ be a family of probability measures over $\X$. Then $\mathcal{F}$ is compact for the narrow topology if, and only if, it is a tight family, \emph{i.e.}, for all $\varepsilon>0$, there is a compact set $K$ such that
    \[
        \mu(\X\setminus K) < \varepsilon \quad \text{ for all $\mu \in \mathcal{F}$.}    
    \]
\end{theorem}

Actually, the set $\mathcal{K} = \mathscr{C}_b(\X)$ is not the minimal set for which we can define a weak topology that yields the narrow convergence. This is clear since we can always approximate functions in $\mathscr{C}_b(\X)$ with Lipschitz functions, but we can even construct a countable set of test functions yielding the narrow convergence. 
\begin{proposition}[{\cite[Chapter~5]{Ambrosio2008GigliSavare}}]\label{prop:criterion_narrow_conv_countable}
    There exists a countable set $\mathcal{K} = {\left(f_k\right)}_{k \in \mathbb{N}}$ of Lipschitz functions such that, for every sequence of measures ${\left(\mu_n\right)}_{n \in \mathbb{N}}$, ${\left(\mu_n\right)}_{n \in \mathbb{N}}$ converges narrowly to some $\mu$ if, and only if, 
    \[
        \int_\X f_k\dd \mu_n \xrightarrow[n \to \infty]{} \int_\X f_k\dd\mu \quad
        \text{ for all $k \in \mathbb{N}$}.     
    \]
\end{proposition}

Another feature of the narrow topology is its continuity properties. For instance, the integral of a l.s.c.\ function is also l.s.c.,\ which is particularly useful for our analysis considering the energies $\mathcal{L}$ and $\mathcal{H}$. We summarize the results we will require as follows.

\begin{lemma}\label{lemma.lsc_narrow}
    Let $\X$ be a Polish space. The following hold:
\begin{itemize}
    \item \emph{\cite[Lemma~3.5]{Santambrogio2021Cucker}}  If ${\left(\mu_n\right)}_{n\in\mathbb{N}} \subset \mathscr{P}(\X)$ converges to $\mu$ in the narrow topology, then the sequence of product measures ${\left(\mu_n\otimes \mu_n\right)}_{n\in\mathbb{N}}$ converges narrowly to $\mu \otimes \mu$ in $\mathscr{P}(\X)$.
    \item \emph{\cite[Proposition~7.1]{santambrogio2015optimal}} If $F\colon \X \to \mathbb{R}\cup\{+\infty\}$ is l.s.c., then so is 
    \[
        \mathscr{P}(\X) \ni \mu \mapsto \int_\X F \dd \mu,
    \]
    in particular the functional $\nu \mapsto \mathcal{L}(\nu) + \mathcal{H}(\nu,\nu)$, defined in~\eqref{eq:energies_pairwise}, is lower semi-continuous. 
\end{itemize}
\end{lemma}

Now consider a pair of Polish spaces $(\X, d_\X)$ and $(\Y, d_\Y)$ and let $\X \ni x \mapsto \nu^x \in \mathscr{P}(\Y)$ be a measure-valued map. 
\begin{definition}\label{def.measurable_family_measures}
	We say ${\left(\nu^x\right)}_{x\in \X}$ is \emph{measurable} if for any Borel set $B \subset \Y$, the function $x\mapsto \nu^{x}(B)$ is Borel-measurable.
\end{definition}

Now, given $\mu \in \mathscr{P}(\X)$ and a measurable family ${\left(\nu^x\right)}_{x\in \X}$, we can define a new probability measure $\gamma \in \mathscr{P}(\X\times \Y)$ in the product space through duality as
\[
    \int_{\X\times\Y} f(x,y)\dd\gamma(x,y)
    \eqdef 
    \int_\Y\left(\int_\X f(x,y)\dd\nu^x(y)\right)\dd\mu(x),
\]
and we use the notation $\gamma = \mu\otimes\nu^x$. It turns out that all measures $\gamma \in \mathscr{P}(\X\times\Y)$ can be written in this way as a consequence of the \emph{disintegration theorem}, see~\cite[Theorem~1.1.6]{stroock1997multidimensional} for a proof in Polish spaces. 
\begin{theorem}[Disintegration theorem]\label{them:disintegration}
    Let $\X_0$ and $\X_1$ be Polish spaces, and two probability measures $\mu_0 \in \mathscr{P}(\X_0)$ and $\mu_1 \in \mathscr{P}(\X_1)$. If $\pi\colon\X_0\to \X_1$ is a measurable map such that $\pi_\sharp \mu_0 = \mu_1$, then there exists a $\mu_1$-a.e.~uniquely determined Borel family ${\left(\mu_0^{x_1}\right)}_{x_1 \in \X_1} \subset \mathscr{P}(\X_0)$ such that 
    \[
        \mu_0^{x_1}(\X_0\setminus \pi^{-1}(x_1)) = 0 \text{ for $\mu_1$-a.e. $x_1 \in \X_1$},    
    \] 
    and, for every measurable function $f\colon\X_0 \to [0,+\infty]$, it holds that
    \[
        \int_{\X_0}f(x_0)\dd\mu_0(x_0)
        = 
        \int_{\X_1}\left(\int_{\pi^{-1}(x_1)}f(x_0)\dd \mu_0^{x_1}(x_0)\right)\dd\mu_1(x_1).
    \]
    Any such ${\left(\mu_0^{x_1}\right)}_{x_1 \in \X_1}$ is called a disintegration family and we write $\mu_0 = \mu_1 \otimes \mu_0^{x_1}$.
\end{theorem}
Whenever $\gamma \in \Pi(\mu,\nu)$, we apply the previous theorem with $\X_0 = \X\times\Y$, $\X_1 = \X$ and $\pi = \pi_\X$ to write $\gamma = \mu\otimes\nu^x$. We could also consider the disintegration w.r.t.~the second marginal, in which case we write $\gamma = \mu^y\otimes \nu$. One of the most useful, yet simple, applications of the disintegration theorem is the gluing lemma. 

\begin{lemma}[{\cite[Lemma 5.3.2]{Ambrosio2008GigliSavare}}]\label{lemma:gluing_lemma}
    Let $\X_1,\X_2,\X_3$ be Polish spaces, $\gamma_{1,2} \in \mathscr{P}(\X_1\times\X_2)$, and  $\gamma_{1,3} \in \mathscr{P}(\X_1\times\X_3)$ such that
    \[
        {(\pi_{\X_1})}_\sharp \gamma_{1,2} 
        = 
        {(\pi_{\X_1})}_\sharp \gamma_{1,3} = \mu_1.   
    \]
    Then there exists $\gamma_{1,2,3} \in \mathscr{P}(\X_1\times\X_2\times \X_3)$ such that 
    \[
        {(\pi_{\X_1,\X_2})}_\sharp \gamma_{1,2,3} 
        = \gamma_{1,2} 
        \text{ and } 
        {(\pi_{\X_1,\X_3})}_\sharp \gamma_{1,2,3} = \gamma_{1,3}.
    \]
\end{lemma}
\begin{proof}
    The proof consists on taking the disintegration families $\gamma_{1,2} = \gamma_{1,2}^{x_1}\otimes\mu_1(x_1)$, $\gamma_{1,3} = \gamma_{1,3}^{x_1}\otimes\mu_1(x_1)$ and defining the new measure as 
    \[
        \gamma_{1,2,3}
        \eqdef 
        \int_{\X_1} 
        \gamma_{1,2}^{x_1}\otimes\gamma_{1,3}^{x_1}\dd\mu_1(x_1). \qedhere
    \] 
\end{proof}

\subsection{Random probability measures and their weak topologies}\label{subsec:random_probability}
We will also use in this work the notion of random probability measure. The simplest example of this kind of object is a sequence of empirical measures, that is, given an i.i.d.\ sample of random variables ${\left(X_i\right)}_{i \in \mathbb{N}}$, we define the measures 
\[
    \mub_N \eqdef \frac{1}{N}\sum_{i = 1}^N \delta_{X_i}.   
\]
Clearly, for each realization of the random variables, we obtain a different discrete measure. For a random sample of agents ${\left(X_i\right)}_{i \in \mathbb{N}}$, we will describe a profile of strategies with the measures 
\[
    \gammab_N = 
    \frac{1}{N}\sum_{i = 1}^N \delta_{X_i}\otimes\nu_i^{X_i}, 
\]
where $\nu_i^{X_i} \in \mathscr{P}(\Y)$ represents the strategy, possibly in mixed plays, of player $i$. In general, a random measure is defined as follows. 

\begin{definition}\label{def:random_measure}
    Given a probability space $(\Omega, \mathcal{F}, \mathbb{P})$ and a Polish space $\X$, a \emph{random measure} $\mub$ is a map from $\Omega$ into the space of Radon measures on $\X$,
    \[
        \mub \colon \Omega \ni \omega \mapsto \mub(\omega) \in \mathscr{M}_b(\X),
    \]
    which is measurable for the Borel $\sigma$-algebra defined with respect to the narrow topology, in duality with $\mathscr{C}_b(\X)$. We let $\mathscr{M}_\Omega(\X)$ denote the space of all random measures, and $\mathscr{P}_\Omega(\X)$ is the convex subset of $\mathscr{M}_\Omega(\X)$ consisting of all $\mathscr{P}(\X)$-valued random probability measures. 
\end{definition}

Given $\mub \in \mathscr{P}_\Omega(\X)$, the map 
\[
    \mathscr{C}_b(\X) \ni \phi \mapsto 
    \mathbb{E}\left[
        \int_\X \phi \dd \mub(\omega)
    \right],
\]
is a bounded linear map over $\mathscr{C}_b(\X)$, so from Riesz' representation theorem this defines a non-random measure via duality, the \emph{expectation measure} $\mathbb{E}\mub \in \mathscr{P}(\X)$, as 
\begin{equation}\label{eq:expectation_measure}
    \int_\X \phi \dd \mathbb{E}\mub 
    \eqdef 
    \mathbb{E}\left[
        \int_\X \phi \dd \mub(\omega)
    \right]. 
\end{equation}
In particular, a random measure can be identified with a non-random measure if, and only if, it coincides with its expectation almost surely. 

The Glivenko--Cantelli law of large numbers, also known as the Glivenko--Cantelli theorem~\cite{dudley1969speed}, states that the empirical measures $\mub_N$ converge in the narrow topology to $\mu$ with probability $1$. Hence, in order to give a topology to $\mathscr{P}_\Omega(\X)$, the first naive candidate would be to consider $\mathbb{P}$-a.s.~convergence of the random measures in the narrow topology. However, this topology would not be metrizable, and it also does not enjoy good compactness properties as Prokhorov's theorem~\cite{dudley2002analysis_probability}. For these reasons, we consider the narrow topology in $\mathscr{P}_{\Omega}(\X)$. 
\begin{definition}\label{def:narrow_topology_random}
    We say that $f\colon\Omega \times \X \to \mathbb{R}$ is a \emph{random bounded continuous function}, and we let $C_\Omega(\X)$ denote the class of all such functions, if\footnote{
        In~\cite{Aubin2009Set} the random continuous functions are also called Carath\'eodory integrands.
    }
    \begin{enumerate}
        \item $x \mapsto f(\omega, x) \in \mathscr{C}_b(\X)$ almost surely; 
        \item $\omega \mapsto f(\omega, x)$ is $\mathcal{F}$-measurable for all $x \in \X$;
        \item $\omega \mapsto \norm{f(\omega,\cdot)}_{L^\infty(\X)}$ is integrable with respect to $\mathbb{P}$.
    \end{enumerate}
    The \emph{narrow topology of random measures} is then the weakest topology that makes 
    \[
        \mathscr{P}_\Omega(\X) 
        \ni \mub \mapsto 
        \mathbb{E}_{\mathbb{P}} 
        \left[
            \int_X f(\omega, x) \dd (\mub(\omega))(x) 
        \right]
        \text{ continuous for all $f \in C_{\Omega}(\X)$.}     
    \]
\end{definition}
A standard application of Lebesgue's dominated convergence theorem gives that $\mathbb{P}$-a.s.~narrow convergence is stronger than the narrow topology of $\mathscr{P}_\Omega(\X)$. In addition, the weaker notion of narrow convergence in $\mathscr{P}_\Omega(\X)$ enjoys compactness properties analogous to Prokhorov's theorem~\ref{thm:Prokhorov}, see~\cite[Theorem~4.29]{crauel2002random}.

\begin{theorem}[Random Prokhorov's theorem]\label{thm:random_Prokhorov}
    A family of random measures $\mathcal{F} \subset \mathscr{P}_\Omega(\X)$ is pre-compact for the narrow topology of random measures, if and only if, it is tight: for any $\varepsilon>0$ there is a compact set $K_\varepsilon$ such that 
    \[
        \mathbb{E}\left[
            \mub(\X \setminus K_\varepsilon)
        \right]    
        \le \varepsilon
        \text{ for every $\mub \in \mathcal{F}$}.
    \]
\end{theorem}

\subsection{\texorpdfstring{$\Gamma$}{Gamma}-convergence}\label{subsec:gamma_convergence}
The key idea in order to prove the convergence of certain Nash equilibria, associated with potential games, to Cournot--Nash equilibria is to exploit the fact that one can obtain such objects through the minimization of a family of energies indexed by the number of players. The main argument consists in showing that these energies converge in the sense of $\Gamma$-convergence to the potential function that describes Cournot--Nash equilibria. This notion of convergence is defined as follows. 
\begin{definition}\label{def.gamma_convergence}
    Let $(\X,d_\X)$ be a complete metric space. We say that a sequence of functionals $\mathscr{F}_N \colon \X \to \mathbb{R}\cup \{+\infty\}$ \emph{$\Gamma$-converges} to $\mathscr{F}$ if  
    \begin{itemize}
        \item \underline{$\Gamma\text{-}\liminf$:} for every sequence $x_{N} \cvstrong{N \to \infty}{d_\X} x$ in $\X$, it holds that
        \[
            \mathscr{F}(x) \le \liminf_{N \to \infty} \mathscr{F}_N(x_N).    
        \] 
        \item \underline{$\Gamma\text{-}\limsup$:} for every $x \in \X$, there is a sequence $x_{N} \cvstrong{N \to \infty}{d_\X} x$ such that
        \[
            \limsup_{N \to \infty} \mathscr{F}_N(x_N) \le \mathscr{F}(x),           
        \]
        ${\left(x_N\right)}_{N\in\mathbb{N}}$ is called the \emph{recovery sequence} of $x$.
    \end{itemize}
\end{definition} 

The notion of $\Gamma$-convergence was introduced by De Giorgi in order to have good properties concerning the limits of minimizers of variational problems, see for instance~\cite{dal1993introduction}. In this sense, the fundamental property that makes it interesting is the fact that cluster points of minimizers of a sequence of minimizers of $\mathscr{F}_N$, which $\Gamma$-converges to $\mathscr{F}$, are minimizers of $\mathscr{F}$. Indeed, let ${\left(x_N\right)}_{N \in \mathbb{N}}$ be a sequence of minimizers of ${\left(\mathscr{F}_N\right)}_{N \in \mathbb{N}}$ converging to $x$. For an arbitrary $x'\in \X$, let $(x_N')_{N \in \mathbb N}$ be a corresponding recovery sequence. Then it follows that
\begin{equation*}
    \begin{aligned}
        \mathscr{F}(x) & 
        \le \liminf_{N \to \infty} \mathscr{F}_N(x_N) & \quad & \text{by the $\Gamma\text{-}\liminf$ inequality} \\ 
         & 
         \le \liminf_{N \to \infty} \mathscr{F}_N(x'_N) & & \text{by the minimality of $x_N$}\\
         & 
         \le \limsup_{N \to \infty} \mathscr{F}_N(x'_N)
        \le \mathscr{F}(x') & & \text{since $x'_N$ is a recovery sequence of $x'$.}
    \end{aligned}
\end{equation*}
As $x'$ was an arbitrary point of $\X$, it follows that $x$ is a minimizer of $\mathscr{F}$.

Equivalently, given a sequence of functionals $\mathscr{F}_N$, we define the lower and upper $\Gamma$-limits, respectively, as
    \begin{equation}
        \begin{aligned}
            \Gamma\text{-}\liminf \mathscr{F}_N(x)
            &\eqdef
            \inf 
            \left\{
                \liminf_{N \to \infty} 
                \mathscr{F}_N(x_N) : 
                x_N \cvstrong{N \to \infty}{} x
            \right\},\\  
            \Gamma\text{-}\limsup \mathscr{F}_N(x)
            &\eqdef
            \inf 
            \left\{
                \limsup_{N \to \infty} 
                \mathscr{F}_N(x_N) : 
                x_N \cvstrong{N \to \infty}{} x
            \right\}.
        \end{aligned}
    \end{equation}
    From~\cite[Proposition~1.28]{braides2002gamma}, both $\Gamma$ upper and lower limits are lower semi-continuous and $\mathscr{F}_N \cvstrong{N \to \infty}{\Gamma} \mathscr{F}$ if and only if $\Gamma\text{-}\liminf \mathscr{F}_N = \Gamma\text{-}\limsup \mathscr{F}_N = \mathscr{F}$.

\section{Potential structure and stability of the value function}\label{sec:potential_CounotNash}
In this section, our objective is twofold. First, we extend the results of Blanchet and Carlier about the potential structure for Cournot--Nash equilibria, allowing for individual costs $c$ that are l.s.c.~instead of continuous. We then show a stability result of the value function w.r.t.~the fixed marginal $\mu$.

\subsection{Potential structure for Cournot--Nash equilibria}
The goal of this section is to characterize equilibria in the sense of Definition~\ref{def:equ_CournotNash} as critical points of an energy functional. We assume here that the optimization problem a player of type $x \in \X$ tries to solve among a mean field of plays $\nu \in \mathscr{P}(\Y)$ is given by
\begin{equation}\label{eq:Phi_general}
    \min_{ y \in \Y} \Phi(x, y, \nu) \eqdef c(x,y) + \frac{\delta \mathcal{E}}{\delta \nu}(\nu)(y), 
\end{equation}
where $c$ is l.s.c.~and the second term is the first variation of an energy $\mathcal{E}\colon\mathscr{P}(\Y) \to \mathbb{R}$, which is defined below.

\begin{definition}\label{def.first_var_criticalpoint}
    We say that a functional $\mathscr{F}$ defined over the probability measures $\mathscr{P}(\X)$ over a Polish space $\X$ admits a \emph{first variation} at $\mu_0 \in \text{dom}\mathscr{F}$ if there exists a measurable function $f\colon\X \to \mathbb{R}$ such that, for every $\mu$ in the domain of $\mathscr F$, the function $f$ is $\mu - \mu_0$ integrable and
    \begin{equation}\label{eq:first_var}
        \frac{\dd}{\dd \varepsilon}
        \Big{|}_{\varepsilon=0^+}
        \mathscr{F}(\mu_0 + \varepsilon(\mu-\mu_0))
        = 
        \inner{f, \mu - \mu_0}
        \eqdef 
        \int_{\X} f \dd (\mu - \mu_0),
    \end{equation}
    and we write $f = \displaystyle \frac{\delta \mathscr{F}}{\delta \mu}(\mu_0)$. In addition, we say that $\mu_0$ is a \emph{critical point} of $\mathscr{F}$ if 
    \[
        \inner{\frac{\delta \mathscr{F}}{\delta \mu}(\mu_0), \mu - \mu_0} \ge 0
        \text{ for all $\mu \in \dom\mathscr{F}$}.
    \]
\end{definition}

The notion of first variation defined above is classical and has been used in many works, for instance, in optimal transport and mean field games (see, e.g., \cite[Definition~2.2.1]{Cardaliaguet2019Master}, \cite[Definition~5.43]{Carmona2018ProbabilisticI}, or \cite[Chapter~7, page~200]{santambrogio2015optimal}). It is clear that the first variation of a functional $\mathscr F$ cannot be unique, since the sum of a constant to a function $f$ satisfying~\eqref{eq:first_var} will still satisfy the same relation, as the integration is taken against $\mu - \mu_0$, which integrates to $0$. It is, however, unique up to a constant.

For the rest of this paragraph, we fix $\mu \in \mathscr P(\X)$ and we let $\mathcal{E}$ be an l.s.c.~functional over $\mathscr{P}(\Y)$ admitting a first variation. We consider the energy
\begin{equation}\label{eq:energy_lifted_general}
    \mathcal{J}(\gamma) 
    \eqdef 
    \begin{dcases}
        \int_{\X\times \Y} c(x,y)\dd \gamma + \mathcal{E}(\nu),
        & \text{ if } \gamma \in \Pi(\mu,\nu) \text{ for some } \nu \in \mathscr P(\Y),\\ 
        +\infty,
        & \text{ if } \gamma \not\in \mathscr{P}_\mu(\X\times \Y),
    \end{dcases}
\end{equation}
so that the individual cost $\Phi$ from \eqref{eq:Phi_general} reads
\[
    \Phi(x,y,\nu) = \frac{\delta \mathcal{J}}{\delta \gamma}(\gamma)    
    \text{ for $\nu = {(\pi_\Y)}_\sharp\gamma$.}
\]

Our goal is to show that critical points of this energy are Cournot--Nash equilibria. Notice, however, that satisfying the equilibrium condition~\eqref{eq:CN} is independent of having a finite social cost~\eqref{eq:finite_social_cost_condition}: for instance, from an economic perspective, we might have ``bad equilibria'' representing a society with infinite poverty, for instance if a non-negligible part of the population is infinitely poor. In order to avoid taking into account these situations and ensuring that \eqref{eq:finite_social_cost_condition} may hold true, we make the following definition.

\begin{definition}\label{def.distribution_finite_social_cost}
    A measure $\varrho \in \mathscr{P}(\Y)$ is a \emph{distribution of finite social cost} for the distribution $\mu$ if there is a function $\kappa \in L^1(\mu)$ such that for $\mu$-a.e.~$x\in\X$ there is $y_x \in \Y$ satisfying 
    \[
        \Phi(x,y_x,\varrho) \le \kappa(x).
    \]
\end{definition}

The main result of this section is the following.
\begin{theorem}\label{thm:potential_structure_cournot_nash}
    Assume that $\mathcal{E}$ takes values in $\mathbb R_+ \cup \{+\infty\}$ and admits a l.s.c.~first variation with compact sublevel sets, so that $\Phi$ is defined as in~\eqref{eq:Phi_general}. Let $\gamma \in \mathscr{P}_\mu(\X\times \Y)$, then the following assertions are equivalent: 
    \begin{enumerate}
        \item $\gamma$ is a critical point of the functional $\mathcal{J}$ from~\eqref{eq:energy_lifted_general} and $\nu = {\left(\pi_\Y\right)}_\sharp\gamma$ is a distribution of finite social cost;
        \item $\gamma$ is a Cournot--Nash equilibrium with finite social cost, in the sense of Definition~\ref{def:equ_CournotNash}.
    \end{enumerate}
\end{theorem}
\begin{proof}
    The proof is inspired by the arguments in~\cite[Theorem~4.5.1]{arjmand2022thesis} and \cite[Theorem~3.20]{MazantiVariational} for the case of an abstract Lagrangian Mean Field Game and~\cite[Appendix~A]{liu2023mean}, however we assume much less structure and better exploit the compactness of the level sets to avoid some measure theoretic issues. 

    First suppose that $\gamma \in \mathscr{P}_{\mu}(\X\times\Y)$ is a critical point and define the function
    \[
        \phi(x) \eqdef \inf_{\Y}\Phi(x,\cdot,\nu).  
    \]
    It follows that $\phi$ is Borel-measurable since it is lower semi-continuous as we prove next. Take $x_k \cvstrong{k \to \infty}{}x$ such that $\liminf \phi(x_k)$ is finite, otherwise there is nothing to prove, and assume, up to the extraction of a subsequence, that the $\liminf$ is a limit. Consider $y_k \in \argmin \Phi(x_k,\cdot,\nu)$ so that 
    $(\Phi(x_k, y_k,\nu))_{k \in \mathbb N}$ is uniformly bounded by some constant $\ell$. Therefore, as $c \ge 0$, it holds that 
    \[
        {\left(y_k\right)} _{k \in \mathbb{N}} \subset 
        \left\{
            \frac{\delta \mathcal{E}}{\delta \nu}(y) \le \ell
        \right\},
    \]
    which is a compact set. Up to another extraction, we may assume that $y_k \to y$ for some $y \in \Y$, so that the lower semi-continuity of $\Phi$ gives 
    \[
        \phi(x) \le \Phi(x,y,\nu) \le \liminf_{k \to \infty} \Phi(x_k,y_k,\nu) 
        = \liminf_{k \to \infty} \phi(x_k).    
    \]
    
    To prove that $\gamma$ is a Cournot--Nash equilibrium, it suffices to show that the set
    \[
        A = 
        \left\{
            (x,y)\in\X\times\Y: 
            \substack{
                \displaystyle
                \phi(x) < \Phi(x,y,\nu)
            }
        \right\}
    \]
    is $\gamma$-negligible. Suppose this is not the case. Our goal is to construct a Borel measurable selection of the argmin operator, that is, a Borel function $T\colon\X \to \Y$ such that
    \[
        T(x) \in  
        \argmin_{\Y} \Phi(x,\cdot,\nu)
        \text{ for all $x\in X$}.
    \]
    From~\cite[Theorem~1]{brown1973measurable} it holds that if $E \subset \X \times \Y$ is a Borel set with the property that $E_x \eqdef \{y \in \Y : (x,y) \in E\}$ is $\sigma$-compact for all $x \in \pi_\X(E)$, then there is a Borel measurable selection $T \colon \pi_\X(E) \to \pi_\Y(E)$. And from~\cite[Corollary~1]{brown1973measurable}, the measurable selection of the argmin operator can be obtained since $A$ is a Borel set, as $\phi$ and $\Phi$ are Borel measurable, and the sub-level sets of the first variation of $\mathcal{E}$ are compact, so that 
    \[
        \Y = 
        \bigcup_{n \in \mathbb{N}}
        \left\{
            y : \frac{\delta \mathcal{E}}{\delta \nu}(y) \le n
        \right\},
        \text{ is $\sigma$-compact. }
    \]
    
    We define the transportation plan given by
    \[
        \bar \gamma
        \eqdef 
        \gamma \mres \left(\X\times\Y \setminus A\right) 
        + 
        {(\pi_\X, T\circ \pi_\X)}_\sharp \gamma \mres A.  
    \]
    Recalling that $\Phi$ is precisely the first variation of $\mathcal{J}$ evaluated at $\gamma$, we have
    \begin{align*}
        0 \le 
        \inner{
            \frac{\delta\mathcal{J}}{\delta\gamma}(\gamma), \bar\gamma - \gamma
        }
        &= 
        \int_{\X\times\Y}
        \Phi(\bar x,\bar y,\nu)\dd\bar\gamma
        -
        \int_{\X\times\Y}
        \Phi(x,y,\nu)\dd\gamma\\ 
        &=
        \int_{A}
        \underbrace{
            \left(
            \Phi(x,T(x),\nu) - \Phi(x,y,\nu)
        \right)
        }_{ < 0}
        \dd\gamma \le 0.
    \end{align*}
    This contradicts the fact that $\gamma(A) > 0$, and we conclude that $\gamma$ is a Cournot--Nash equilibrium. 

    To prove that it is an equilibrium of finite social cost, since $\nu$ is a distribution of finite social cost, take a function $\kappa$ as in Definition~\ref{def.distribution_finite_social_cost}. By definition, $\mu$-a.e.~we have that $\phi(x) \le \kappa(x)$ so that
    \[
        \int_{\X\times\Y}\Phi(x,y,\nu)\dd\gamma 
        = 
        \int_\X\phi(x)\dd \mu 
        \le 
        \int_\X\kappa(x)\dd\mu < +\infty,
    \]
    yielding the conclusion.

    Conversely, suppose that $\gamma$ is a Cournot--Nash equilibrium of finite social cost. From Definition~\ref{def:equ_CournotNash}, it follows that 
    \[
        \int_{\X\times \Y} \phi(x) \dd \gamma = \int_{\X\times\Y} \Phi(x,y,\nu)\dd \gamma.        
    \]
    Hence, for any other admissible transportation plan $\bar\gamma \in \mathscr{P}_\mu(\X\times\Y)$, it holds that 
    \begin{align*}
        \int_{\X\times\Y} \Phi(x,y,\nu)\dd\bar\gamma
        &\ge 
        \int_{\X\times\Y} \phi(x)\dd\bar\gamma 
        = 
        \int_\X \phi(x)\dd\mu
        = 
        \int_{\X\times\Y} \Phi(x,y,\nu)\dd\gamma. 
    \end{align*}
    From the fact that $\Phi(x,y,\nu) = \displaystyle \frac{\delta\mathcal{J}}{\delta\gamma}(\gamma)$, we conclude that $\gamma$ is a critical point of $\mathcal{J}$. Using the same measurable selection argument used above, we can construct a measurable map $x\mapsto y_x \in \argmin \Phi(x,\cdot, \nu)$ which satisfies Definition~\ref{def.distribution_finite_social_cost} with $\kappa = \phi$.
\end{proof}

The previous result gives a stronger characterization of Cournot–Nash equilibria in a general setting, as discussed in the introduction. For the rest of this work, and specially for the proof of convergence of Nash to Cournot--Nash equilibria, we shall concentrate on a case where $\mathcal{E}$ is given as the sum of a linear and an interaction term, as in~\cite{blanchet2016optimal,Santambrogio2021Cucker, MazantiVariational} and \cite[Chapter~4]{arjmand2022thesis}. That is, when $\mathcal{E}$ can be written as follows
\begin{equation} 
        \mathcal{E}(\nu) 
        = 
        \mathcal{L}(\nu) + \mathcal{H}(\nu,\nu), 
        \text{ where }
        \mathcal{L}(\nu) = \int_\Y L\dd\nu \text{ and } 
        \mathcal{H}(\nu,\nu) = \int_{\Y\times\Y} H\dd \nu\otimes\nu.
\end{equation} 
In this case, the lifted energy from~\eqref{eq:energy_lifted_general} becomes
\begin{equation}\label{eq:energy_lifted}
    \mathcal{J}(\gamma) 
    \eqdef
    \int_{\X\times\Y}c\dd\gamma 
    + 
    \mathcal{L}(\nu)
    +
    \mathcal{H}(\nu,\nu).
\end{equation}
As the integral of l.s.c.~functionals, both $\mathcal{L}$ and $\mathcal{H}$ are l.s.c.~as functionals over $\mathscr{P}(\Y)$ thanks to Lemma~\ref{lemma.lsc_narrow}. Since the level sets of $\mathcal{L}$ are compact, we would be able to prove existence of minimizers for $\mathcal{J}$, were it not for the term $\mathcal{H}$ that can be $+\infty$, for instance if $H$ diverges in the diagonal.

The game with a continuum of players that we are interested in is described by the cost 
\begin{equation}\label{eq.case_paper}
    \Phi \colon \left\{
    \begin{aligned}
        \X\times\Y\times\mathscr{P}(\Y) & \to \mathbb{R}_+ \cup \{+\infty\} \\ 
        (x,y,\nu) & \mapsto c(x,y) + L(y) + 2\int_\Y H(y,y')\dd \nu(y'),
    \end{aligned}\right.
\end{equation}
whenever $H$ is symmetric. We make the following assumptions on these functionals throughout the rest of this work. 
\begin{hypothesis}
    \item\label{Hypo-atomless} $\mu \in \mathscr{P}(\X)$ is atomless.
    \item\label{Hypo-LH} $c\colon \X \times \Y \to \mathbb{R}_+\cup\{+\infty\}$, $L\colon \Y \to {\mathbb R}_+\cup\{+\infty\}$, and $H\colon \Y \times \Y \to {\mathbb R}_+\cup\{+\infty\}$ are lower semi-continuous.
    \item\label{Hypo-H-symmetric} $H$ is symmetric, \emph{i.e.}, $H(y, \widetilde y) = H(\widetilde y, y)$ for every $(y, \widetilde y) \in \Y \times \Y$. 
    \item\label{Hypo-compact} $L$ has compact sub-level sets, \emph{i.e.}, for every $\kappa > 0$, the set $\{L\leq \kappa\}$ is compact.
\end{hypothesis}

Let us make some remarks on the above hypotheses.
\begin{remark}\label{remark.atomless_non_restrictive}
    Hypothesis~\ref{Hypo-atomless} is not restrictive, since if $\mu$ has atoms, we can work in the lifted space 
    $
        \X' = [0,1] \times \X$ and consider $\mu' = \mathcal{L}^1\mres[0,1]\otimes \mu.
    $
    Then we can formulate a new game with $c$ replaced by $c'(x',y) = c(\pi_\X(x'),y)$, which remains l.s.c.~in the product space $\X'\times\Y$. This new game will then satisfy all hypotheses~\ref{Hypo-atomless}--\ref{Hypo-compact}.
\end{remark}
\begin{remark}
    It is not restrictive either to consider $L$ and $H$ as only functions of the strategies since we can replace $\Y$ by $\widetilde{\Y} = \X \times \Y$, and $c(x, y)$ with $\widetilde c(x, (x', y)) = c(x,y)$ if $x' = x$ and $+\infty$ otherwise.
\end{remark}

In this particular case, besides the previous characterization, we have the following result which is slightly stronger since the first variation does not need to be well-defined for all $\nu$.
\begin{theorem}\label{thm:minimizer_is_equilibria}
    Consider the case that $\mathcal{E} = \mathcal{L} + \mathcal{H}$ as in~\eqref{eq:energies_pairwise} and assume  that
        \[
            C \eqdef 
            \inf_{\mathscr{P}_\mu(\X\times \Y)} \mathcal{J}
            < +\infty.
        \]
        Then there exists a minimizer $\bar \gamma$ for
        $\displaystyle
            \mathcal{J}.  
        $
        
        In addition, if for all $n \in \mathbb{N}$ there exists $0 < C_n < +\infty$ and a measurable family ${\left(\nu_n^y\right)}_{y\in \Y} \subset \mathscr{P}(\Y)$ such that for all $y,y' \in \Y$ and $(x,y,\nu) \in \dom \Phi$ it holds that
        \begin{equation}\label{Hypo-technical}
            \mathcal{H}(\nu_n^y, \nu_n^{y'}) 
            \le C_n(1 + L(y) + L(y')), 
            \text{ and }
            \int_\Y \Phi(x,z,\nu) \dd \nu_n^y(z)
            \cvstrong{n \to \infty}{} \Phi(x,y,\nu),
        \end{equation}
        then any minimizer of $\mathcal{J}$ is a Cournot--Nash equilibrium.
\end{theorem}
\begin{proof}
    Notice first that the existence of a minimizer follows easily from the direct method of the calculus of variations. Indeed, since $L$ has compact sub-level sets, any minimizing sequence is tight, so that we can extract a convergent sub-sequence by means of Prokhorov's theorem, whose limit is a minimizer thanks to lower semi-continuity. Fix a minimizer $\bar \gamma$ and its second marginal $\bar \nu \eqdef {(\pi_\Y)}_\sharp\bar \gamma$. If $\mathcal{J}$ admits a first variation w.r.t.~$\bar \gamma$, the result follows with Theorem~\ref{thm:potential_structure_cournot_nash} since minimizers are critical points provided that the first variation of $\mathcal E$ exists at them. However, the limit defining the first variation only makes sense in this case along directions $\gamma \in \text{dom}\mathcal{J}$ whose second marginal $\nu$ is such that both $\mathcal{H}(\nu, \nu)$ and $\mathcal{H}(\nu, \bar\nu)$ are finite. Indeed, in this case we have 
    \begin{multline*}
        \mathcal{J}(\bar \gamma + \varepsilon(\gamma - \bar \gamma)) 
        = \mathcal{J}(\bar \gamma) +
        \varepsilon \int_{\X\times \Y} (c(x,y) + L(y)) \dd (\gamma - \bar \gamma)
        \\ + 2 \varepsilon \int_{\Y \times \Y} H \dd\bar \nu \otimes (\nu - \bar \nu) 
        + \varepsilon^2 \left(\mathcal{H}(\bar \nu, \bar \nu) + \mathcal{H}(\nu,\nu) + 2\mathcal{H}(\nu, \bar \nu)\right).
    \end{multline*}
    As a result, recalling that in this case we have $\Phi(x,y,\nu) = \displaystyle c(x,y) + L(y) + 2\int_\Y H(y,y')\dd \nu(y')$, for any minimizer $\bar \gamma$, we have 
    \begin{equation}\label{eq.variation_eq}
        \int_{\X\times \Y} \Phi(x,y,\bar \nu) \dd (\gamma - \bar \gamma) \ge 0.
    \end{equation}
    
    Next, define the function
    \[
        \phi(x) 
        \eqdef 
        \inf_{y \in \Y} 
        c(x,y) + L(y) + 2\int_\Y H(y,\bar y) \dd \bar\nu(\bar y).
    \]
    Since $c,L \ge 0$, we have that 
    \[
        \int_\X \phi(x) \dd \mu(x)
        \le 2\mathcal{J}(\bar\gamma) \le 2C,
    \]
    so that $\phi(x)$ is finite for $\mu$-a.e.~$x\in \X$. For any such $x$, we can once again apply the direct method and the arguments from Theorem~\ref{thm:potential_structure_cournot_nash} to obtain a measurable selection $x \mapsto y_x$ of the argmin operator defining $\phi(x)$. 
    In addition, for each $y_x$, we consider the sequence ${\left(\nu_n^{y_x}\right)}_{n \in \mathbb{N}}$ satisfying~\eqref{Hypo-technical} for $\mu$-a.e. $x \in \X$. As a result, defining the measure 
    \[
        \nu_n(A) \eqdef 
        \int_\X \nu_n^{y_x}(A) \dd \mu(x),
        \text{ for any Borel set } A \subset \Y,
    \]
    it holds that 
    \begin{align*}
        \mathcal{H}(\nu_n, \nu_n) 
        &= 
        \int_{\X\times \X} 
        \mathcal{H}\left(
            \nu_n^{y_x}, \nu_n^{y_{x'}}
        \right)
        \dd \mu\otimes \mu(x,x')\\
        &\le 
        C_n\int_{\X \times \X}
        \left(
            1 + L(y_x) + L(y_{x'})
        \right)\dd \mu\otimes\mu\\ 
        &\le 
        C_n
        \left(
            1 + 2\int_\X \phi(x)\dd \mu(x)
        \right)
        < +\infty. 
    \end{align*}

    Hence we can define a new transportation plan to serve as competitor, given by 
    \[
        \gamma_n \eqdef 
        \frac{1}{2}\mu\otimes\nu_n^{y_x} 
        + 
        \frac{1}{2}\bar \gamma,
    \]
   which can be used in~\eqref{eq.variation_eq}. Introducing a disintegration for $\bar \gamma = \mu\otimes \bar \nu^x$, we obtain from~\eqref{eq.variation_eq} that
    \[
        \int_\X 
        \left[ 
            \int_\Y \Phi(x,y,\bar \nu) \dd \bar\nu^x(y)
            - 
            \int_\Y \Phi(x,y,\bar \nu) \dd \nu_n^{y_x}(y)
        \right] \dd \mu(x)
        \le 0.
    \]
    Since the second term in the integral over $\X$ converges $\mu$-a.e.~to $\Phi(x,y_x,\bar \nu) = \phi(x)$, we obtain from the Dominated Convergence Theorem that
    \[
        \int_\X 
        \underbrace{\left[ 
            \int_\Y \Phi(x,y,\bar \nu) \dd \bar\nu^x(y)
            - 
            \phi(x)
        \right]}_{\ge 0} 
        \dd \mu(x)
        \le 0.
    \]
    We conclude that for $\bar \gamma$-a.e.~$(x,y)$ it holds that
    \[
        \Phi(x,y,\bar\nu) = \phi(x),
    \]
    so that $\bar \gamma$ is a Cournot--Nash equilibrium.
\end{proof}

In the previous theorem, the condition that the infimum is finite is non-trivial and should be verified for each problem. Imposing further conditions on $\mathcal{E}$, such as strict convexity, this can be verified as done in~\cite{blanchet2016optimal}.

\subsection{On the finiteness of the infimum}
In the case of a potential of the form \eqref{eq.case_paper}, we can characterize the cases where there the infimum is finite, and hence when we have existence, with a measure defined with the individual transportation cost $c$ and the interaction energy $H$ as follows: For $K \subset \Y$ compact, define 
\begin{equation}
\label{eq:def-capp}
    \capp_{c,H}(K) 
    \eqdef 
    {\left(
        \inf_{\varrho \in \mathscr{P}_{c,\mu}(K)} 
        \int_{\mathcal{Y}\times \mathcal{Y}} H\dd \varrho\otimes\varrho
    \right)}^{-1},
\end{equation}
where $\mathscr{P}_{c,\mu}(K) \eqdef \displaystyle \left\{ \varrho \in \mathscr{P}(K) : \mathcal{W}_c(\mu,\varrho) < +\infty \right\}$. The capacity of an open set $U \subset \Y$ can then be defined through outer regularity 
\[
    \capp_{c,H}(U)
    \eqdef 
    \sup
    \left\{
        \capp_{c,H}(K) : 
        K \subset U
    \right\},
\]
and for a general set $A$ as the inf of the same quantity among all the open sets $U$ containing $A$. This defines a monotone set function that can be used to characterize when the infimum of $\mathcal{J}$ is finite. 
\begin{lemma}\label{lemma.condition_inf_finite}
    Under hypotheses~\ref{Hypo-LH}--\ref{Hypo-compact}, it holds that
    \[
        \inf \mathcal{J} < +\infty 
        \iff 
        \capp_{c,H}(\{L < +\infty\}) > 0.
    \]
\end{lemma}
\begin{proof}
    Starting with the direct implication, suppose that there exists $\gamma \in \mathscr{P}_\mu(\X\times\Y)$ such that $\mathcal{J}(\gamma)<+\infty$. In particular, letting $\nu$ denote the second marginal of $\gamma$, it follows that $\mathcal{W}_c(\mu,\nu) < +\infty$ and $\supp \nu \subset \{L < +\infty\}$. It then follows that
    \[
        \capp_{c,H}(\{L < +\infty\}) \ge {\mathcal{J}(\gamma)}^{-1} > 0.
    \]
    Conversely, if $\capp_{c,H}(\{L<+\infty\}) > 0$, there is some $N \in \mathbb{N}$ such that 
    \[
    \capp_{c,H}(\{L\le N\}) > 0.
    \]
    Hence there is a measure $\varrho$ concentrated over the compact set $\{L \le N\}$ such that $\mathcal{W}_c(\mu,\varrho) < +\infty$ and $\mathcal{H}(\varrho,\varrho) < +\infty$. Taking $\gamma$ as an optimal transportation plan between $\mu$ and $\varrho$ gives that $\mathcal{J}(\gamma) < +\infty$.
\end{proof}

The previous Lemma seems almost tautological, but in some particular cases there are strong results in the literature that characterize exactly which are the sets with positive capacity. In Examples~\ref{example.lagrangianMFG} and~\ref{example.lagrangianMFG} we treat two models whose particular properties allow to verify the capacity criterion from Lemma~\ref{lemma.condition_inf_finite} as well as the additional condition~\eqref{Hypo-technical}.  

\begin{example}\label{example.lagrangianMFG}
    In the Lagrangian mean field game of \cite[Chapter~4]{arjmand2022thesis}, the interaction term is shown to be bounded by the individual cost, that is, there is a constant $C>0$ such that for all $\nu \in \mathscr{P}(\Y)$ it holds that $\mathcal{H}(\nu,\bar\nu) \le C(1 + \mathcal{L}(\nu) + \mathcal{L}(\bar\nu))$, which trivializes the capacity condition since $\mathcal{L}$ is not identically $+\infty$. Condition~\eqref{Hypo-technical} is also easily verified by taking $\nu_n = \delta_y$.
\end{example}
\begin{example}\label{example.frostman_condition}
    Consider now a simpler case where $\X = \Y = \mathbb{R}^d$,
    \[
        c \in \mathscr{C}_b(\mathbb{R}^d \times \mathbb{R}^d), \ 
        L \in \mathscr{C}_b(\mathbb{R}^d)
        \text{ and }
        H(y,\bar y) = |y - \bar y|^{-\alpha} \text{ for some $0 < \alpha < d$.}
    \]
    The condition that $c$ is bounded implies that the set $\mathscr{P}_{c,\mu}(\mathbb{R}^d)$ is equal to $\mathscr{P}(\mathbb{R}^d)$, since the optimal transportation problem $\mathcal{W}_c(\mu,\varrho)$ is finite for any probability measure $\varrho$. In this way, the capacity condition becomes 
    \[
        \capp_\alpha(\{L<+\infty\}) > 0, 
    \]
    where $\capp_\alpha$ denotes the Riesz capacity, defined over $\mathbb{R}^d$ by \eqref{eq:def-capp} taking $H(y,\bar y) = |y - \bar y|^{-\alpha}$ and $c \equiv 0$. In this case, Frostman's Lemma (see~\cite[Chapter~4.3]{falconer2004fractal}, \cite[Appendix~B]{ponce2016elliptic}, or the original thesis of Frostman~\cite{frostman1935potential}) gives a charaterization of sets with strictly positive $\alpha$-capacity in $\mathbb{R}^d$. Indeed, for a general Borel set $A \subset \mathbb{R}^d$, it holds that
    \[
        d_H(A) = \inf\left\{ 
            s \ge 0 : \capp_s(A) = 0
        \right\},
    \]
    where $d_H(A)$ denotes the Hausdorff dimension of the set $A$. 
    
    It follows that, in order to satisfy the capacity condition, it suffices to verify that the set $\{L < +\infty\}$ is of Hausdorff dimension bigger than $\alpha$. Similarly, to verify~\eqref{Hypo-technical}, since $c,L$ are continuous, for fixed $x,y$ and $\nu$ such that $\displaystyle \int_{\mathbb{R}^d} |y - y'|^{-\alpha}\dd \nu(y') < +\infty$, it can be checked that the map 
    \[
       z \mapsto 
       \int_{\mathbb{R}^d} |z - y'|^{-\alpha}\dd \nu(y')
    \]
    is continuous and bounded in a neighborhood of $y$. 

    Using Frostman's Lemma, there is a sequence $\nu_n \in \mathscr{P}\left(\overline{B_{r_n/4}(0)}\right)$ satisfying $\mathcal{H}(\nu_n,\nu_n)<+\infty$, resulting that $\nu_n \cvweak{n \to \infty} \delta_0$. To construct the family ${(\nu_n^y)}_{y \in \Y}$, first we consider the grid 
    \[
        r_n \mathbb{Z}^d, \text{ where } r_n = 2{n}^{-1/\alpha}. 
    \]
    For a given $y \in \mathbb{R}^d$, if $\dist(y,\mathbb{Z}^d)$ is attained by $y_n$, we define 
    \[
        \nu_n^y(A) \eqdef \nu_n(y_n + A).
    \]
    Hence, if $y,y'$ are both projected onto the same $y_n$, we obtain that 
    \[
        \mathcal{H}(\nu_n^{y},\nu_n^{y'})
        = 
        \mathcal{H}(\nu_n,\nu_n)
        <+\infty,
    \]
    otherwise, we have that
    \[
        \mathcal{H}(\nu_n^{y},\nu_n^{y'}) 
        \le n.
    \]
    In addition, since for any $y \in \mathbb{R}^d$ any sequence of projections onto $r_n\mathbb{Z}^d$ converges to $y$, it holds that $\nu_n^y \cvweak{n \to \infty} \delta_y$
    and~\eqref{Hypo-technical} follows by setting $C_n \eqdef \max\{
    n, \mathcal{H}(\nu_n,\nu_n)\}$.
\end{example}

Although it seems hard to verify the additional condition~\eqref{Hypo-technical}, the following example indicates that it is actually sharp
\begin{example}
    Take $c, L = 0$, $\Y = \{0,1\}$ and the pair-wise interaction $H$ given by 
    \[
        H(0,0)= 0, \ H(0,1)= H(1,0)= -1, \ H(1,1)= +\infty.
    \]
    In this case, the potential energy becomes the minimization among measures of the form 
    \[
        \nu = \nu_0 \delta_0 + \nu_1\delta_1, \qquad (\nu_0, \nu_1) \in [0, 1]^2 \text{ with } \nu_0 + \nu_1 = 1,
    \]
    of the energy given by 
    \[
        \mathcal{H}(\nu, \nu) 
        = 
        -2\nu_0\nu_1 + \infty \nu_1^2.
    \]
    This energy becomes $+\infty$ whenever $\nu_1 > 0$, so that the only minimizer must be given by $\nu_\star = \delta_0$. On the other hand, $\nu_\star$ can never be a Nash equilibrium since 
    \[
        0 \notin 
        \argmin_{y \in \{0,1\}} H(0,y) = \{ 1 \}.
    \]
    
    Observe that~\eqref{Hypo-technical} can never be satisfied with this choice of $H$, but replacing $H(1,1) = M$, with $0 < M < +\infty$, it actually holds that $H \in \mathscr{C}_b$ when endowed with the discrete topology and one can verify with a simple computation that the unique minimizer of the potential energy yields a Nash equilibrium. 
\end{example}

\subsection{Stability of the value function}
\label{sec:stability-value-function}
In this paragraph our primary goal is to show, assuming the additional hypothesis~\ref{Hypo-gluing_operator} below, that we have the estimate 
\begin{equation}\label{eq:stability_value_function}
    \left|  
        \inf_{\gamma \in \mathscr{P}_{\mu_0}(\X\times\Y)}
        \mathcal{J}
        -
        \inf_{\gamma \in \mathscr{P}_{\mu_1}(\X\times\Y)}
        \mathcal{J} 
    \right|
    \le 
    C W_1(\mu_0,\mu_1)
    \text{ for some $C>0$,}
\end{equation}
where $\mu_0, \mu_1 \in \mathscr{P}(\X)$ are two distribution of agents and 
\[
    W_1(\mu_0,\mu_1)
    \eqdef 
    \min_{\gamma \in \Pi(\mu_0,\mu_1)} 
    \int_{\X\otimes\X} 
    d_\X(x_0,x_1)\dd \gamma
\]
is the $1$-Wasserstein, or Kantorovitch--Rubinstein, distance, see~\cite{villani2009optimal}.

The stability result~\eqref{eq:stability_value_function} is of independent interest. In particular, it gives some regularity for the value function $\mu \mapsto \inf_{\gamma \in \mathscr{P}_{\mu}(\X\times\Y)} \mathcal{J}$ in the Wasserstein topology and allows to estimate the value with arbitrary approximations of a given reference distribution $\mu$, instead of for only empirical measures covered by our $\Gamma$-convergence results from Section~\ref{sec:convergence}. However, it is important to emphasize that, contrarily to $\Gamma$-convergence, this result says nothing about the convergence of minimizers for the problems referent to a sequence $\mu_n$ converging to $\mu$. Another potential application would be the construction of $\varepsilon$-equilibria, which could be done with a suitable approximation of a given reference measure. However, this cannot be done for instance in an Euclidean space with an empirical measure since we expect that
\[
    W_1(\mu_N,\mu) \approx N^{-1/d} 
\]
if $\mu \in \mathscr{P}(\mathbb{R}^d)$ and $\mu_N$ is a suitable sequence of empirical measures solving the optimal quantization problem~\cite{merigot2021non}. Recalling that the cost minimized by each player behaves as $N \times \mathcal{J}(\gamma_N)$, this estimate does not give information on whether the minimization of $\mathcal{J}$ over $\mathscr{P}_{\mu_N}(\X\times\Y)$ yields $\varepsilon$-Nash equilibria for the $N$-player game. But it might be useful in this direction if one can carefully construct a sequence $\mu_N$ of estimates for $\mu$ with faster convergence rates.

To prove~\eqref{eq:stability_value_function}, we will exploit the \emph{gluing method}, recently introduced in~\cite{liu2023mean}. This method depends on the existence of a  gluing operator as described in the following assumption:
\begin{hypothesis}[start=5]
    \item\label{Hypo-gluing_operator} There exists an operator $\mathcal{G}\colon \X\times \X\times \Y \to \Y$ such that: 
    \begin{itemize} 
        \item $\mathcal{G}$ is consistent: for every $y \in \Y$ and $x \in \X$, it holds that $\mathcal{G}(x,x,y) = y$;
        \item there exists a positive constant $C>0$ satisfying
        \begin{equation}\label{eq:gluing_operator}
            \begin{aligned}
                c\left(x_1,\mathcal{G}(x_1,x_0,y)\right) 
                &\le 
                c(x_0,y) + C d_\X(x_1,x_0),\\
                L\left(\mathcal{G}(x_1,x_0,y)\right)
                &\le 
                L(y) + C d_\X(x_1,x_0),\\  
                H(\mathcal{G}(x_1,x_0,y),\mathcal{G}(\tilde x_1,\tilde x_0,\tilde y))
                &\le 
                H(y,\tilde y) +
                C \left(
                    d_\X(x_1,x_0) + d_\X(\tilde x_1,\tilde x_0)
                \right).
            \end{aligned}
        \end{equation}
        for any pairs $x_0,x_1\in \X$ and $y,\tilde y \in \Y$.
    \end{itemize}
\end{hypothesis}
Essentially, hypothesis~\ref{Hypo-gluing_operator} says that there is an operator that, given some player of type $x_0$ choosing the play $y$, any other player of type $x_1$ can choose a strategy $\mathcal{G}(x_1, x_0, y)$ paying a perturbation, of order $d_\X(x_0,x_1)$, of the cost paid by the first player. With this assumption we can prove the following result.

\begin{lemma}[Gluing method]\label{lemma:gluing_method}
    Let $\mu_0,\mu_1$ be probability measures in $\mathscr{P}(\X)$ and $\gamma_0 \in \mathscr{P}_{\mu_0}(\X\times \Y)$. Under the hypothesis~\ref{Hypo-gluing_operator}, there exists a measure $\gamma_1 \in \mathscr{P}_{\mu_1}(\X\times\Y)$ such that 
    \[
        \mathcal{J}(\gamma_1)
        \le 
        \mathcal{J}(\gamma_0) + 4C W_1(\mu_0,\mu_1),
    \]
    where $W_1$ denotes the Kantorovitch--Rubinstein distance. 
\end{lemma} 
\begin{proof}
    Given measures $\mu_0,\mu_1 \in \mathscr{P}(\X)$ and let $\pi_{1,0} \in \Pi(\mu_1, \mu_0)$ be an optimal transportation plan between $\mu_0$ and $\mu_1$, \emph{i.e.}
    \[
        \int_{\X_1\times \X_0}d_\X(x_1, x_0) \dd \pi_{1,0}(x_1,x_0) = W_1(\mu_1, \mu_0),    
    \]
    where $\X_0$ and $\X_1$ are identical copies of the space $\X$.

   Let $\Gamma \in \mathscr{P}(\X_1 \times \X_0 \times \Y)$ denote the gluing of $\pi_{1,0}$ and $\gamma_0$ in the sense of Lemma~\ref{lemma:gluing_lemma}, so that ${\left(\pi_{\X_0,\Y}\right)}_{\sharp} \Gamma = \gamma_0$ and ${\left(\pi_{\X_1,\X_0}\right)}_{\sharp} \Gamma = \pi_{1,0}$. The measure $\gamma_1$ is then defined as $\gamma_1 := {\left(\pi_{\X_1},\mathcal{G}\right)}_{\sharp} \Gamma$. It follows from these definitions that
   \begin{align*}
        \mathcal{J}(\gamma_0) 
        &= 
        \int_{\X_0\times \Y} \left(c + L + H\right)\dd \Gamma\otimes \Gamma, \\
        \mathcal{J}(\gamma_1) 
        &= 
        \int_{\X_1\times \Y} \left(c + L + H\right)
        \dd 
        \left(
            {\left(\pi_{\X_1},\mathcal{G}\right)}_{\sharp}\Gamma
        \right)
        \otimes 
        \left(
            {\left(\pi_{\X_1},\mathcal{G}\right)}_{\sharp}\Gamma
        \right).
   \end{align*}
   Using the definition of the gluing operator from~\ref{Hypo-gluing_operator}, we get the following estimates
    \begin{align*}
        \mathcal{J}(\gamma_1)
        &=
        \int_{\X_1\times \Y} 
        (
            c(x_1, \mathcal{G}(x_1,x_0,y)) +
            L(\mathcal{G}(x_1,x_0,y)) \\
            & \quad \quad \quad \quad \quad 
            + H(\mathcal{G}(x_1,x_0,y),\mathcal{G}(\bar x_1,\bar x_0, \bar y))
        )
        \dd\Gamma\otimes\Gamma
        \\ 
        &\le 
        \int_{\X_0\times \Y} \left(c + L + H\right)
        \dd\Gamma\otimes\Gamma \\
        & \quad \quad \quad \quad \quad 
        + C \int_{\X_1\times\X_0\times\Y} 
        (3d_\X (x_1,x_0) + d_\X (\tilde x_1,\tilde x_0))\dd \Gamma \otimes \Gamma\\ 
        &=
        \mathcal{J}(\gamma_0) + 4C\int_{\X_1\times\X_0} d_\X (x_1,x_0)\dd \pi_{1,0} \\ 
        &= 
        \mathcal{J}(\gamma_0) + 4C W_1(\mu_0,\mu_1). 
    \end{align*}
    The result follows.
\end{proof}

We are now in position to prove our main result on the stability of the value function.

\begin{theorem}\label{thm:stability_value_function}
    Under the hypothesis~\ref{Hypo-gluing_operator}, the stability inequality~\eqref{eq:stability_value_function} for the value function holds.
\end{theorem}
\begin{proof}
    Let $\gamma_0 \in \mathscr{P}_{\mu_0}(\X\times\Y)$ be optimal for the minimization of $\mathcal J$ over $\mathscr{P}_{\mu_0}(\X\times\Y)$, that is,
    \[
        \mathcal{J}(\gamma_0) = \min_{\mathscr{P}_{\mu_0}(\X\times\Y)}\mathcal{J}.   
    \]  
    Let $\gamma_1$ be the measure obtained from the gluing method in Lemma~\ref{lemma:gluing_method}. It then holds that
    \[
        \inf_{\mathscr{P}_{\mu_1}(\X\times\Y)}\mathcal{J} 
        - 
        \inf_{\mathscr{P}_{\mu_0}(\X\times\Y)} \mathcal{J}
        \le
        \mathcal{J}(\gamma_1) - \mathcal{J}(\gamma_0)
        \le 
        4CW_1(\mu_0, \mu_1), 
    \]
    where $C$ is the constant from~\ref{Hypo-gluing_operator}.
    Changing the roles of $\mu_0$ and $\mu_1$, we conclude. 
\end{proof}

\begin{example}[Back to Example~\ref{example.lagrangianMFG}]
    Let us now give further details for the Lagrangian MFG discussed in Example~\ref{example.lagrangianMFG}. We consider a model where a population of agents tries to reach a target set in minimal time under pairwise interactions.
    
    For simplicity, let $\Omega \subset \mathbb{R}^d$ be a convex set, and let $\Gamma \subset \Omega$ be the target set of the players. In this case $\X = \Omega$ and $\Y = C(\mathbb{R}_+; \Omega)$, the continuous functions with values in $\Omega$. For $\sigma \in \Y$ we set 
    \[
        \tau(\sigma) 
        \eqdef 
        \inf\{
          t\ge 0 : \sigma(t) \in \Gamma  
        \},
    \]
    the minimal time to reach the target and
    \[
        c(x_0,\sigma) = 
        \begin{dcases}
            0,& \text{if $\sigma(0) = x_0$},\\
            +\infty,& \text{ otherwise.}
        \end{dcases}    
    \]
    The individual and interaction energies are given by 
    \begin{align*}
        L(\sigma) &\eqdef 
        \int_0^{\tau(\sigma)}   
        \ell(t, \sigma(t), \dot{\sigma}(t))\dd t + \Psi(\sigma_\tau), \\
        H(\sigma, \bar \sigma)
        &\eqdef 
        \int_0^{\tau(\sigma)\wedge \tau(\bar \sigma)}   
        h(t, \sigma(t), \dot{\sigma}(t), \bar \sigma(t), \dot{\bar \sigma}(t))\dd t.
    \end{align*}
    For simplicity, we assume that if $\dot{\sigma}(t) = 0$, then we have $\ell(t, \sigma(t), \dot{\sigma}(t)) = 0$ and $h(t, \sigma(t), \dot{\sigma}(t), \bar \sigma(t), \dot{\bar \sigma}(t)) = 0$, $\ell$ and $h$ are bounded non-negative functions, and that $\sigma$ remains constant after reaching $\Gamma$ for the first time. 

    In order to have a small perturbation of the energies, the easiest way is to preserve the stopping time, hence given $\sigma$ such that $\sigma(0) = x_0$ we search for a curve of the form 
    \[
        \sigma_{x_1}(t) 
        \eqdef 
        \begin{dcases}
            \left(1 - \frac{t}{t_0}\right)x_1 + \frac{t}{t_0}\sigma(t_0),& 
            \text{ if } t\in[0,t_0],\\ 
            \sigma(t),& \text{ otherwise.}
        \end{dcases}    
    \]
    Therefore, choosing $t_0 \le \min\{\tau_\sigma, |x_0-x_1|\}$ we obtain that
    \[
        L(\sigma_{x_1}(t))
        \le 
        L(\sigma) + \int_0^{t_0}\ell(t, \sigma_{x_1}(t), \dot{\sigma}_{x_1}(t))\dd t 
        \le 
        L(\sigma) + C |x_0-x_1|.    
    \]
    An analogous reasoning for $H$ gives the required result. Hence \ref{Hypo-gluing_operator} is satisfied in this example.
\end{example}

\section{Convergence of Nash to Cournot--Nash equilibria}\label{sec:convergence}

In this section we prove our convergence results of Nash to Cournot--Nash equilibria. For readability, we perform the analysis in both open- and closed-loop cases separately, even if at the cost of some repetition. In this way, both results can be read independently. We start by recalling the definitions of each formulation described in the introduction and discuss them in more details as well. 

First, let us recall the definition of Nash equilibrium and introduce some notation. 
\begin{definition}[Nash equilibrium in $N$-player game]\label{def.Nashequi}
    An \emph{$N$-player game in pure strategies} is a tuple ${\left(g_i, S_i\right)}_{i = 1}^N$ where $S_i$ denotes the space of admissible plays for player $i$ and $g_i$ is a function
\[
    g_i \colon S_i \times S_{-i} \ni (x_i, x_{-i}) \mapsto g_i(x_i, x_{-i}) \in \mathbb{R}
    \text{ where }
    S_{-i} \eqdef \prod_{j \neq i} S_j.
\] 
Given an admissible profile of strategies ${\left(x_j\right)}_{j = 1}^N$, $x_{-i}$ corresponds to the tuple of strategies deprived of $x_i$ and the quantity $g_i(x_i, x_{-i})$ represents the cost of player $i$ choosing $x_i$ given that the remaining players choose $x_{-i}$. A profile $(x_i)_{i = 1}^N$ is called a \emph{Nash equilibrium} if no player has the incentive to deviate from their strategy, \textit{i.e.},
\[
    g_i(x_i,x_{-i}) \le  g_i(x_i',x_{-i}) \text{ for all $x_i'\in S_i$ and all $i=1,\dots,N$}. 
\]
\end{definition}

A game in mixed strategies, or mixed plays, is a tuple ${\left(\bar g_i, \mathscr{P}(S_i)\right)}_{i = 1}^N$, such that 
\[
    \bar g_i(\nu_i, \nu_{-i}) 
    \eqdef 
    \int_{S_i \times S_{-i}}
    g_i(x_i, x_{-i})
    \dd \nu_i\otimes \nu_{-i}(x_i, x_{-i}),
\] 
where $\nu_{-i} \eqdef \nu_1\otimes \dotsb \otimes \nu_{i-1} \otimes \nu_{i+1} \otimes \dotsb \otimes \nu_N$.

In order to define the games in open and closed loop information briefly discussed in the introduction we fix a probability space $(\Omega, \mathcal{F}, \mathbb{P})$ and consider an i.i.d.~sample of agents ${\left(X_i\right)}_{i \in \mathbb{N}}$ with common law $\mu \in \mathscr{P}(\X)$, where the first $N$ elements represent the type of the agents in our $N$-player game.

\subsection{Convergence in the open-loop formulation}\label{sec:conv_open_loop}

\subsubsection{Potential structure in open loop}
In the open-loop formulation, as each player chooses a strategy before having the knowledge of the realization of the sample, the type of player $i$ is better described by the random variable $X_i$ and an admissible strategy must be given by a measurable family ${\left(\nu^x\right)}_{x \in \X} \subset \mathscr{P}(\Y)$, or equivalently by a random probability measure as discussed in Section~\ref{subsec:random_probability}, see also~\cite{crauel2002random,kallenberg2017random}. In other words, instead of choosing a deterministic strategy, given a state $x$, a player chooses a distribution of strategies $\nu^x$. The criterion that each player seeks to minimize is then
\begin{equation*}\label{eq:game_ol_mix}
        \tag{$NP_\Omega$}
        \min_{\nub_i \in \mathscr{P}_{\Omega}(\Y)} {}  
        \mathcal{J}_{\Omega,i}\left(\nub_i, \nub_{-i}\right) 
        \eqdef 
        \mathbb{E}
        \left[
        \int_\Y c(X_i,y)\dd \nu_i^{X_i} + \mathcal{L}\left(\nu_i^{X_i}\right)
        + 
        \frac{2}{N}\sum_{i \neq j}
        \mathcal{H}\left(
            \nu_i^{X_i},\nu_j^{X_j}
        \right)
        \right],
\end{equation*}
where $\nub_i = \nu_i^{X_i}$ for some measurable map ${(\nu_i^x)}_{x \in \X}$.

A profile ${\left(\nub_i\right)}_{i = 1}^N$ is said to be in pure strategies if each $\nub_i$ is a Dirac delta with full probability and can then be described with a map as measures of the form $\nub_i = \delta_{Y_i}$. The formulation in pure strategies can then be expressed as 
\begin{equation}
    \begin{aligned}
        \min_{Y_i \text{ is $X_i$-measurable}} {}  
        \mathcal{J}_{\Omega, i}(Y_i, Y_{-i}) 
        &=
        \mathbb{E}
        \left[
            c(X_i,Y_i)
            +
            L(Y_i)
            + 
            \frac{2}{N}\sum_{i \neq j}\mathcal{H}(Y_i,Y_j)
        \right].
        \end{aligned}
\end{equation}

As introduced in Section~\ref{subsec:random_probability}, we let $\mub_N \in \mathscr{P}_{\Omega}(\X)$ be the random measure obtained via the sample of random variables $\displaystyle \mub_N \eqdef \frac{1}{N}\sum_{i = 1}^N\delta_{X_i}$, and we define the space of random transportation plans induced by a profile of strategies
\begin{multline}\label{eq.plans_induced_strategies}
    \mathscr{P}_{\Omega, N}(\X \times \Y)
    \eqdef 
    \Biggl\{
        \gammab_N 
        = \frac{1}{N}\sum_{i = 1}^N \delta_{X_i}\otimes\nu_i^{X_i}: \\
        {\left(\nu^x_i\right)}_{x \in \X} \text{ is measurable for all $i = 1,\dots,N$}
    \Biggr\},
\end{multline}
where we recall the definition of measurable family of measures from Definition~\ref{def.measurable_family_measures}. Notice that this class of random measures is contained in the set of random measures whose first marginal coincides with $\mub_N$, but with the particularity that the stochasticity in the family of disintegration is restricted to the evaluation of the random variables $X_i$, whereas in this larger class of measures, we could have a stochastic behaviour in the disintegration family as well. 

The potential function in open-loop formulation is defined by
\begin{equation}\label{eq:potential_function_openloop}
    \mathcal{J}_{\Omega, N}(\gammab_N)
    \eqdef
    \begin{dcases}
        \begin{aligned}
            \mathbb{E}\left[
            \int_{\X\times \Y} c \dd \gammab_N +
            \frac{1}{N}\sum_{i = 1}^N \mathcal{L}\left(\nu_i^{X_i}\right)
             \right.\\ 
            \left.
            {} + 
            \frac{1}{N^2}\sum_{i \neq j} 
            \mathcal{H}\left(\nu_i^{X_i}, \nu_j^{X_j}\right)
        \right],
        \end{aligned}
        &
        \text{ if } \gammab_N\in\mathscr{P}_{\Omega,N}(\X\times\Y),\\ 
        +\infty,& \text{ otherwise.}
    \end{dcases} 
\end{equation}
Observe that, since by Hypothesis~\ref{Hypo-atomless} the measure $\mu$ has no atoms, we have $\mathbb{P}\left(\{X_i = X_j\}\right) = 0$ for all $i \neq j$, so that the disintegration families of the plans in the set $\mathscr{P}_{\Omega,N}(\X \times \Y)$ from~\eqref{eq.plans_induced_strategies} are almost surely uniquely defined. In fact, the disintegration theorem (Theorem~\ref{them:disintegration}) gives a canonical bijection between the set of random measures $\mathscr{P}_{\Omega,N}\left(\X\right)$ and the set of strategy profiles admissible for the open-loop formulation. As a consequence, the potential function~\eqref{eq:potential_function_openloop} is well-defined and becomes a suitable candidate for potential function to describe equilibria. Indeed, since $c$, $L$ and $H$ are l.s.c.~and $L$ has compact level sets, it admits minimizers, which induce Nash equilibria for the corresponding game, as we shall prove next. 

\begin{proposition}\label{proposition:NP-potential_open}
    Let 
        \[
                \gammab_N = \frac{1}{N}\sum_{i = 1}^N \delta_{X_i}\otimes\nu_i^{X_i} \in \argmin \mathcal{J}_{\Omega,N}.
        \]
        Then ${\left(\nu_i^{X_i}\right)}_{i=1}^N$ is a Nash equilibrium for the game~\eqref{eq:game_ol_mix}.
\end{proposition}

\begin{proof}
    Let $\nub_j = \nu_j^{X_j}$ and fix some index $i \in \{1,\dots, N\}$ to consider a deviation, so that player $i$ chooses $\bar\nub_i$ instead. Consider the new random transportation plan
    \[
        \bar \gammab_N
        \eqdef 
        \frac{1}{N} \sum_{\substack{j = 1\\ j \neq i}}^N \delta_{X_j}\otimes\nub_j 
        + 
        \frac{1}{N} \delta_{X_i} \otimes \bar\nub_i. 
    \]
    First notice that, from the symmetry of $H$, we have
    \begin{align*}
        \sum_{j\neq k}H(y_j,y_k)
        &= 
        \sum_{\substack{j,k = 1 \\ i \neq j,\, j \neq k,\, i \neq k}}^N H(y_j,y_k)
        + 
        \sum_{\substack{k = 1\\ k \neq i}}^NH(y_i,y_k)
        +
        \sum_{\substack{j = 1\\ j \neq i}}^NH(y_j,y_i)\\ 
        &= 
        \sum_{\substack{j,k = 1 \\ i \neq j,\, j \neq k,\, i \neq k}}^N H(y_j,y_k)
        +
        2\sum_{\substack{j = 1\\ j \neq i}}^NH(y_j,y_i),
    \end{align*}
    so the minimality of $\gammab_N$ gives 
    {\small
    \begin{align*}
        & \frac{1}{N}
        \left(
            \mathbb{E}\left[
            \sum_{\substack{j = 1\\ j \neq i}}^N
            \int_\Y[c(X_j, y_j) + L(y_j)]\dd \nub_j 
            + 
            \frac{1}{N}
            \sum_{\substack{j,k = 1 \\ i \neq j,\, j \neq k,\, i \neq k}}^N
            \mathcal{H}(\nub_j,\nub_k)
            \right]
            + 
            \mathcal{J}_{\Omega, i}(\nub_i, \nub_{-i})
        \right) \\
        & =
        \mathcal{J}_{\Omega, N}(\gammab_N) 
        \le
        \mathcal{J}_{\Omega, N}(\bar \gammab_N) \\
        & =
        \frac{1}{N}
        \left(
            \mathbb{E}\left[
            \sum_{\substack{j = 1\\ j \neq i}}^N
            \int_\Y[c(X_j, y_j) + L(y_j)]\dd \nub_j 
            + 
            \frac{1}{N}
            \sum_{\substack{j,k = 1 \\ i \neq j,\, j \neq k,\, i \neq k}}^N
            \mathcal{H}(\nub_j,\nub_k)
            \right]
            + 
            \mathcal{J}_{\Omega, i}(\bar\nub_i, \nub_{-i})
        \right).
    \end{align*}}%
    Canceling out the repeated terms we obtain that $\mathcal{J}_{\Omega, i}(\nub_i, \nub_{-i}) \le \mathcal{J}_{\Omega, i}(\bar\nub_i, \nub_{-i})$, meaning that the profile $(\nub_1,\dots,\nub_N)$ is a Nash equilibrium. 
\end{proof}

\subsubsection{\texorpdfstring{$\Gamma$}{Gamma}-convergence for the open-loop formulation}

The $\Gamma$-convergence result relies on the characterization of the cluster points of random measures in the set $\mathscr{P}_{\Omega,N}(\X\times \Y)$ as $N \to +\infty$. Indeed, since the stochasticity of this class of measures is contained in the first marginal given by the empirical measure of an i.i.d.\ sample, we can expect that the limit will be a non-random measure. This is proved in the following result.

\begin{lemma}\label{lemma:convergente_random_to_deterministic}
    Let $\left(\gammab_N\right)_{N \in \mathbb N}$ be a sequence of random measures such that $\gammab_N \in \mathscr{P}_{\Omega,N}(\X\times\Y)$ for all $N \in \mathbb{N}$ and such that the corresponding sequence of expectation measures converges in the narrow topology of $\mathscr{P}(\X\times\Y)$, that is
    \[
        \mathbb{E}\gammab_N 
        \cvweak{N \to \infty} 
        \gamma.
    \]
    for some $\gamma \in \mathscr P(\X \times \Y)$.
    Then $\gamma \in \mathscr{P}_\mu(\X\times\Y)$ and $\gammab_N$ converges in the narrow topology of $\mathscr{P}_{\Omega}(\X\times \Y)$ to a random measure $\gammab$ that coincides with $\gamma$ almost surely. 
\end{lemma}
\begin{remark}
    The major difficulty of the following proof comes from the fact that the operation of disintegration, the conditional expectation, is not continuous w.r.t.~weak convergence in general. In order words, if a sequence of measures ${\left(\gamma_N\right)}_{N \in \mathbb{N}}$ converging weakly to $\gamma$ has the disintegration representation $\gamma_N = \mu\otimes \nu_N^x$ and $\gamma = \mu\otimes \nu^x$, it does not hold in general that $\nu_N^x \xrightharpoonup[N \to \infty]{} \nu^x$ for $\mu$~a.e.~$x$. 
\end{remark}
\begin{proof}
    It follows directly from~\eqref{eq.plans_induced_strategies} that the sequence ${\left(\gammab_N\right)}_{N \in \mathbb{N}}$ can be written as 
    \[
        \gammab_N = \frac{1}{N}
        \sum_{i = 1}^N \delta_{X_i}\otimes\nu_{i,N}^{X_i}. 
    \]
    To finish the proof, it suffices to show that, for any real-valued, bounded and $(\Omega,\mathcal{F},\mathbb{P})$-adapted random variable $\Theta$ and $\varphi \in \mathscr{C}_b(\X\times\Y)$, it holds that
    \begin{equation}\label{eq.convergence_rv_Hoeffding}
        \Delta_{N,\Theta} \eqdef 
        \left|
            \mathbb{E}
            \left[
                \Theta \int_{\X\times\Y} \varphi \dd \gammab_N
            \right]
            -
            \mathbb{E}
            \left[
                \Theta
            \right]
            \int_{\X\times\Y} \varphi \dd \mathbb{E}\gammab_N
        \right|
        \cvstrong{N \to \infty}{}
        0,
    \end{equation}
    since then we will have that 
    \[
        \lim_{N \to \infty}
        \mathbb{E}
        \left[
            \Theta 
            \left(
                \int_{\X\times\Y} \varphi \dd(\gammab_N - \gamma)
            \right)
        \right] = 0
    \] 
    for any bounded random variable $\Theta$, meaning that $\gammab = \gamma$ almost surely. 

    For this, we will use \emph{Hoeffding's inequality}, which states that, if $Z_1, \dots, Z_N$ are independent real variables such that $a \le Z_i \le b$ for some real numbers $a, b$ and all $i$ almost surely, then 
    \begin{equation}\label{eq:hoeffding}
        \mathbb{P}\left(
            \left|
                \frac{1}{N}\sum_{i = 1}^N 
                (Z_i - \mathbb{E}Z_i)
            \right|\ge \varepsilon
        \right)
        \le 2\exp\left(-\frac{2N\varepsilon^2}{{(b-a)}^2}\right). 
    \end{equation}
    Notice that we can rewrite 
    \[
        \int_{\X\times\Y} \varphi \dd \gammab_N = 
        \frac{1}{N} \sum_{i = 1}^N \tilde \varphi_{i,N}
        \text{ where }        
        \tilde \varphi_{i,N}
        \eqdef 
        \int_{\Y} \varphi(X_i,y) \dd \nu_{i,N}^{X_i},
    \]
    so that ${\left(\tilde \varphi_{i,N}\right)}_{i=1}^N$ are independent real random variables, $|\tilde \varphi_{i,N}| \le \norm{\varphi}_{L^\infty}$ and 
    \[
        \frac{1}{N}\sum_{i = 1}^N 
        \mathbb{E}\tilde \varphi_{i,N}
        = 
        \int_{\X\times\Y} \varphi \dd \mathbb{E}\gammab_N.
    \] 
    So setting 
    \[
        A_{\varepsilon,N} \eqdef 
        \left\{\left|
        \frac{1}{N} \sum_{i = 1}^N 
        \left(
        \tilde \varphi_{i,N}
        -
        \mathbb{E}\tilde \varphi_{i,N}
        \right)
        \right| \ge \varepsilon \right\},
    \]  
    we can use Hoeffding's inequality to bound the LHS of~\eqref{eq.convergence_rv_Hoeffding}
    \begin{align*}
        \Delta_{N,\Theta} 
        &\le 
        \mathbb{E}
        \left[
            |\Theta|
            \left|
                \frac{1}{N} \sum_{i = 1}^N 
                \left(
                \tilde \varphi_{i,N}
                -
                \mathbb{E}\tilde \varphi_{i,N}
                \right)
            \right|
        \right]\\ 
        &\le 
        \norm{\Theta}_{L^\infty} 
        \int_{A_{\varepsilon,N}} 
            \left|
            \frac{1}{N} \sum_{i = 1}^N 
            \left(
            \tilde \varphi_{i,N}
            -
            \mathbb{E}\tilde \varphi_{i,N}
            \right)
            \right|
            \dd \mathbb{P}
        +
        \norm{\Theta}_{L^\infty} \varepsilon\\ 
        &\le 
        2\norm{\Theta}_{L^\infty}\norm{\varphi}_{L^\infty} 
        \mathbb{P}(A_{\varepsilon,N})
        + 
        \norm{\Theta}_{L^\infty} \varepsilon\\ 
        &\le 
        4\norm{\Theta}_{L^\infty}\norm{\varphi}_{L^\infty} 
        \exp\left(-\frac{N\varepsilon^2}{2\norm{\varphi}_{L^\infty}^2}\right)
        + 
        \norm{\Theta}_{L^\infty} \varepsilon
    \end{align*}
    Choosing $\varepsilon = N^{-1/3}$, we get that $\Delta_{N,\Theta} \cvstrong{N \to \infty}{} 0$. We conclude that $\gammab = \gamma$ almost surely. 
\end{proof}

\begin{remark}
    In fact we have shown that $\gammab_N$ converges to $\gamma$ in the much stronger topology of narrow convergence $\mathbb{P}$-almost surely.
\end{remark}

The previous lemma is the crucial observation that allows the passage of the limit of a sequence of stochastic variational problems to a deterministic one, as we shall see in the following $\Gamma$-convergence result. 
\begin{theorem}\label{theorem:gamma_conv_openloop}
    Given an i.i.d.~sample ${\left(X_i\right)}_{i\in\mathbb{N}}$ with law $\mu$, let $\mub_N \in \mathscr{P}_\Omega(\X)$ be the associated sequence of empirical random measures. Let $\mathcal{J}_{\Omega,N}$ be the sequence of potential functionals defined in~\eqref{eq:potential_function_openloop}. Then it holds that
    \[
        \mathcal{J}_{\Omega,N}
        \cvstrong{N \to \infty}{\Gamma}
        \mathcal{J}_\Omega,
    \]
    with
    \[
        \mathcal{J}_\Omega(\gammab)
        \eqdef 
        \begin{dcases}
            \mathcal{J}(\gamma),& \text{ if } \gammab = \gamma \text{ a.s.\ for some } \gamma \in \mathscr{P}_\mu(\X\times\Y),\\ 
            +\infty,& \text{ otherwise},
        \end{dcases}
    \]
    where the $\Gamma$-convergence is in $\mathscr{P}_{\Omega}(\X\times\Y)$ equipped with the narrow topology of random probability measures.
\end{theorem}
\begin{proof}
    Starting with $\Gamma$-$\liminf$, consider a sequence ${\left(\gammab_N\right)}_{N \in \mathbb{N}}$ converging to $\gammab$ in the narrow topology of random measures. From Lemma~\ref{lemma:convergente_random_to_deterministic}, it follows that $\gammab$ is actually a non-random measure $\gamma \in \mathscr{P}_{\mu}(\X\times\Y)$. Without loss of generality we assume that
    \[
        \liminf_{N \to \infty} \mathcal{J}_{\Omega,N}(\gammab_N) < \infty,    
    \]
    otherwise there is nothing to prove. Then, up to taking a subsequence attaining the $\liminf$, one can assume that $\mathcal{J}_{\Omega,N}(\gammab_N) \le C$ for some constant $C > 0$ and all $N \in \mathbb{N}$, so in particular $\gammab_N \in \mathscr{P}_{\Omega, \mu_N}(\X\times\Y)$. 

    For an arbitrary $M>0$, define the truncated interaction energy as
    \[
        \mathcal{H}^M(\nu,\nu) \eqdef \int_{\Y\times\Y} H^M \dd \nu\otimes \nu, 
        \text{ where }
        H^M \eqdef H\wedge M,
    \]
    and the truncated total energies $\mathcal{J}^M$ and $\mathcal{J}_{\Omega,N}^M$ as in~\eqref{eq:energy_lifted} and~\eqref{eq:potential_function_openloop} by replacing $\mathcal{H}$ with $\mathcal{H}^M$. Then it follows from Fubini's Theorem that 
    \begin{align*}
        \mathcal{J}_{\Omega,N}^M(\gammab_N) 
        &=
        \mathbb{E}
        \left[
            \frac{1}{N}\sum_{i = 1}^N
            \int_\Y c(X_i, y) + L(y) \dd \nu_N^{X_i}
            +
            \frac{1}{N^2}\sum_{i \neq j}
            \mathcal{H}^M(\nu_N^{X_i}, \nu_N^{X_j})
        \right] \displaybreak[0] \\ 
        &= 
        \mathbb{E}
        \left[
            \int_{\X \times \Y} c + L \dd \gammab_N
            +
            \int_{\Y \times \Y} 
            H^M \dd \gammab_N\otimes \gammab_N
        \right] \\
        & \quad \quad \quad \quad \quad
        - 
        \frac{1}{N^2}
        \sum_{i = 1}^{N}
        \mathbb{E}
        \left[
            \mathcal{H}^M\left(\nu_N^{X_i}, \nu_N^{X_i}\right)
        \right]
         \displaybreak[0] \\ 
        &=
        \mathcal{J}^M\left(\mathbb{E}\gammab_N\right)
        - 
        \frac{1}{N^2}\sum_{i = 1}^N
        \underbrace{
            \mathbb{E}\left[
            \mathcal{H}^M(\nu^{X_i}_N,\nu^{X_i}_N)
        \right]
        }_{
            \le M
        }
        \ge 
        \mathcal{J}^M\left(\mathbb{E}\gammab_N\right)
        - 
        \frac{M}{N},
    \end{align*}
    where the last inequality was obtained from the fact that $\mathcal{H}^M$ is bounded by $M$. For any fixed $M>0$, the sum on the right-hand side above vanishes as $N \to \infty$ and hence since $\mathbb{E}\gammab_N \cvweak{N \to \infty}\gamma$, the lower semi-continuity of $\mathcal{J}^M$, as a consequence for instance of Lemma~\ref{lemma.lsc_narrow}, gives that
    \begin{align*}
        \liminf_{N \to \infty}  
        \mathcal{J}_{\Omega,N}(\gammab_N) 
        \ge
        \liminf_{N \to \infty}  
        \mathcal{J}^M\left(\mathbb{E}\gammab_N\right)
        \ge 
        \mathcal{J}^M(\gamma).
    \end{align*}
    Noticing that from the monotone convergence theorem $\mathcal{H}(\nu, \nu) = \displaystyle \sup_{M>0}\mathcal{H}^M(\nu, \nu)$, we get
    \[
        \liminf_{N \to \infty}  
        \mathcal{J}_{\Omega,N}(\gammab_N) 
        \ge 
        \sup_{M>0}
        \mathcal{J}^M(\gamma)
        = 
        \mathcal{J}(\gamma),
    \]
    and the result follows. 
    
    To prove the $\Gamma$-limsup, it suffices to construct recovery sequences only for non-random transportation plans $\gamma \in \mathscr{P}_{\mu}(\X\times\Y)$. For any such measure, consider its disintegration representation as $\gamma = \mu\otimes \nu^x$, and define a recovery sequence as
    \[
        \gammab_N \eqdef \mub_N\otimes \nu^x.
    \]  
Therefore it follows directly from Lemma~\ref{lemma:convergente_random_to_deterministic} that $\gammab_N \cvweak{N \to \infty} \gamma$, since by construction we have that $\mathbb{E}\gammab_N = \gamma$ for all $N \in \mathbb{N}$. 
    
    Next, for each $N \in \mathbb{N}$, a simple computation yields 
    \begin{align*}
        \mathcal{J}_{\Omega,N}(\gammab_N) 
        & =
        \mathbb{E}\Biggl[
            \frac{1}{N}\sum_{i = 1}^N
            \left(
                \int_\Y c(X_i,y) \dd \nu^{X_i}(y)
                + 
                \int_\Y L(y) \dd \nu^{X_i}(y) 
            \right) \\
        & \hphantom{= \mathbb E \Biggl[} 
            + 
            \frac{1}{N^2}\sum_{i\neq j} \int_{\Y\times\Y} H \dd\nu^{X_i}\otimes \nu^{X_j}        
        \Biggr] \displaybreak[0]\\
        & =
        \frac{1}{N}\sum_{i = 1}^N
        \int_{\X}
        \left(
            \int_\Y c(x_i,y) \dd \nu^{x_i}(y)
            + 
            \int_\Y L(y) \dd \nu^{x_i}(y) 
        \right)\dd \mu(x_i)\\
        & \hphantom{= {}} +
        \frac{1}{N^2}\sum_{i\neq j} 
        \int_{\X\times \X}
        \left(
            \int_{\Y\times\Y} H \dd\nu^{x_i}\otimes \nu^{x_j}
        \right)\dd \mu\otimes\mu(x_i,x_j)  \displaybreak[0]
        \\ 
        & =
        \frac{1}{N}\sum_{i = 1}^N
        \int_{\X\times \Y} 
        (c + L) \dd \gamma 
        + 
        \frac{1}{N^2}\sum_{i\neq j} 
        \int_{\Y\times\Y} H \dd \nu \otimes \nu  \displaybreak[0] \\
        & = 
        \mathcal{J}(\gamma) - \frac{1}{N}\int_{\Y\times \Y}H \dd \nu\otimes\nu \le \mathcal{J}(\gamma).
    \end{align*}
    Taking the $\limsup$ as $N \to \infty$, the result follows. 
\end{proof} 

Now we use the properties of $\Gamma$-convergence along with the random Prokhorov's theorem, Theorem~\ref{thm:random_Prokhorov}, to show that cluster points of equilibria for the $N$-player game are Cournot--Nash equilibria in the sense of Definition~\ref{def:equ_CournotNash}.

\begin{theorem}\label{theorem:nash2cournotnash_openloop}
    Under the same assumptions of Theorem~\ref{thm:minimizer_is_equilibria}, if ${\left(\gammab_N\right)}_{N \in \mathbb{N}}$ is a sequence of minimizers of $\mathcal{J}_{\Omega,N}$, then there exists a subsequence ${\left(\gammab_{N_k}\right)}_{k \in \mathbb{N}}$ such that
    \[
        \gammab_{N_k} \cvweak{N_k \to \infty} \gamma \in \mathscr{P}_\mu(\X\times\Y),
    \]
    in the narrow topology of $\mathscr{P}_\Omega(\X\times\Y)$, and in addition $\gamma$ is a Cournot--Nash equilibrium in the sense of Definition~\ref{def:equ_CournotNash}. 

    Assuming in addition that $H \in \mathscr{C}_b(\Y\times \Y)$, for any sequence of Nash equilibria ${\left(\gammab_N\right)}_{N \in \mathbb{N}}$ for the game~\eqref{eq:game_ol_mix}, that is 
    \begin{equation}\label{eq:nash_random_plan}
        \gammab_N 
        \eqdef 
        \frac{1}{N}\sum_{i = 1}^N \delta_{X_i}\otimes \nu_{i,N}^{x} \in 
        \mathscr{P}_{\Omega, \mub_N}(\X\times\Y),
    \end{equation}
    converging to $\gamma$ in the narrow topology of $\mathscr{P}_\Omega(\X\times\Y)$, it holds that 
    $\gamma \in \mathscr{P}_\mu(\X\times\Y)$, and it is a Cournot--Nash equilibrium in the sense of Definition~\ref{def:equ_CournotNash}.
\end{theorem}
\begin{proof}
   To prove the first assertion, we know from the properties of $\Gamma$-convergence that 
    \begin{equation}\label{eq:convergence_infs}
        \inf_{\mathscr{P}_{\Omega,N}(\X\times\Y)} 
        \mathcal{J}_{\Omega,N} 
        \xrightarrow[N \to \infty]{} \inf_{\mathscr{P}_{\mu}(\X\times\Y)}\mathcal{J} \eqdef C < + \infty.
    \end{equation}
    
    Note that, since for each $N\in\mathbb{N}$ the functional $\mathcal{J}_{\Omega,N}$ is l.s.c.~with compact sub-level sets, a sequence $(\gammab_N)_{N \in \mathbb N}$ of minimizers of $\mathcal J_{\Omega, N}$ does exist. By Theorem~\ref{theorem:gamma_conv_openloop}, if such a sequence has a cluster point, then the latter must minimize $\mathcal{J}$. Hence, to finish the proof of the first part of the statement, it suffices to obtain such a cluster point. This will be done with the version of Prokhorov's theorem for random measures, Theorem~\ref{thm:random_Prokhorov}, which states that a sequence of random measures is sequentially compact in the narrow topology if and only if it is tight.

    As $\mub_N \xrightharpoonup[N \to \infty]{}\mu$ in the narrow topology of random measures, it is a tight family, from the Prokhorov's theorem, so for any $\varepsilon > 0$ there exists a compact set $K_{\X,\varepsilon} \subset \X$ such that 
    \[
        \mathbb{E}\left[
            \mub_N(\X \setminus K_{\X,\varepsilon})
        \right]    
        < \frac{\varepsilon}{2}.
    \]

    From~\eqref{eq:convergence_infs} we get that, for $N$ large enough, 
    \[
        \mathbb{E}\left[
            \int_\Y L \dd \nub_N
        \right]
        \le 
        2 C,
    \]
    so that, for some $\varepsilon> 0$ we obtain from Markov's inequality that 
    \[
        \mathbb{E}
        \left[
            \nub_N\left(
                \left\{
                    L \le \frac{4C}{\varepsilon}
                \right\}
            \right)
        \right]    
        \le 
        \frac{\mathbb{E}\left[
            \displaystyle
            \int_\Y L \dd \nub_N
        \right]}{4C/\varepsilon}
        \le \frac{\varepsilon}{2}.
    \]
    Since $L$ has compact level sets, we set $K_{\Y,\varepsilon} = \left\{ L \le 4C/\varepsilon\right\}$ and set $K_\varepsilon \eqdef K_{\X,\varepsilon}\times K_{\Y,\varepsilon}$, so that
    \begin{align*}
        \mathbb{E}\left[
            \gammab_N\left(
                \X\times\Y\setminus K_\varepsilon
            \right)
        \right]
        &\le
        \mathbb{E}\left[
            \gammab_N\left(
                (\X\setminus K_{\X,\varepsilon})\times\Y
            \right)
        \right]
        +
        \mathbb{E}\left[
            \gammab_N\left(
                \X\times(\Y\setminus K_{\Y,\varepsilon})
            \right)
        \right]\\ 
        &=
        \mathbb{E}\left[
            \mub_N\left(
                \X\setminus K_{\X,\varepsilon}
            \right)
        \right]
        +
        \mathbb{E}\left[
            \nub_N\left(
                \Y\setminus K_{\Y,\varepsilon}
            \right)
        \right]
        <\varepsilon.
    \end{align*}
    We conclude that the sequence of random measures $\left(\gammab_N\right)_{N \in \mathbb N}$ is tight, and hence admits a convergent subsequence in the narrow topology of $\mathscr{P}_\Omega(\X\times\Y)$. As discussed above, from Lemma~\ref{lemma:convergente_random_to_deterministic}, the limit of this subsequence belongs to $\mathscr{P}_{\mu}(\X\times\Y)$ and minimizes $\mathcal{J}$. From Theorem~\ref{thm:minimizer_is_equilibria}, $\gamma$ is a Cournot--Nash equilibrium in the sense of Definition~\ref{def:equ_CournotNash}.

    To prove the second assertion, let $\gammab_N$ be defined as in~\eqref{eq:nash_random_plan} and $\gamma$ a limit point, which belongs to $\mathscr{P}_\mu(\X\times\Y)$ from Lemma~\ref{lemma:convergente_random_to_deterministic}. Since in this case we have $H \in \mathscr{C}_b(\Y\times\Y)$, the second marginal $\nu = {(\pi_\Y)}_\sharp\gamma$ is of finite social cost, the functional $\mathcal{J}$ admits a first variation and our goal is to use the characterization obtained in Theorem~\ref{thm:potential_structure_cournot_nash} by verifying that $\gamma$ is a critical point of $\mathcal{J}$, \emph{i.e.}, for any $\bar \gamma \in \mathscr{P}_{\mu}(\X\times\Y)$ we verify that
    \[
        \inner{\frac{\delta \mathcal{J}}{\delta \gamma}, \bar \gamma - \gamma}
        =
        \int_{\X\times\Y}
        \left(
            c(x,y) + L(y)
            +
            2\int_\Y H(y,\bar y)\dd \nu(\bar y)
        \right)
        \dd (\bar \gamma - \gamma)(x,y)
        \ge 0.
    \] 
    From Theorem~\ref{thm:potential_structure_cournot_nash}, this will show that $\gamma$ is a Cournot--Nash equilibrium. 

    Fix some $\bar \gamma \in \mathscr{P}_{\mu}(\X\times\Y)$, and recall the recovery sequence obtained from the $\Gamma$-convergence proof; consider a disintegration family $\bar \gamma = \mu \otimes \bar \nu^x$ so that 
    \[
        \bar \gammab_N 
        \eqdef 
        \frac{1}{N}\sum_{i = 1}^N \bar \nu^{X_i} \cvweak{N \to \infty} \gamma. 
    \]
    We consider a unilateral deviation of player $i$ with the alternative strategy $\bar \nu^{X_i}$, to the profile $\left(\nu_{1,N}^{X_1}, \dots, \nu_{N,N}^{X_N}\right)$. Since the latter is a Nash equilibrium in mixed strategies, we get that 
    $
        \mathcal{J}_{\Omega,i}(\bar \nu^{X_i}, \nu_{-i,N}^{X_{-i}}) 
        \ge 
        \mathcal{J}_{\Omega,i}( \nu_{i,N}^{X_i}, \nu_{-i,N}^{X_{-i}})
    $, 
    for $\mathcal{J}_{\Omega,i}$ defined in~\eqref{eq:game_ol_mix}. This can be rewritten as 
    \begin{align*}
        &\mathbb{E}
        \left[
            \int_\Y c(X_i, y) + L(y) \dd \bar \nu^{X_i}
            + 
            \frac{2}{N} \sum_{j \neq i} 
            \int_{\Y\times\Y} H \dd \bar\nu^{X_i}\otimes \nu_{j,N}^{X_j}
        \right]\\ 
        &\quad \quad \ge 
        \mathbb{E}
        \left[
            \int_\Y c(X_i, y) + L(y) \dd \nu_{i,N}^{X_i}
            + 
            \frac{2}{N} \sum_{j \neq i} 
            \int_{\Y\times\Y} H \dd \nu_{i, N}^{X_i} \otimes \nu_{j,N}^{X_j}
        \right].
    \end{align*}
    Let us define the measures
    \[
        \gammab_{N,-i} \eqdef \frac{1}{N}\sum_{j\neq i} \delta_{X_j}\otimes \nu_{j,N}^{X_j}, 
        \text{ and }
        \nub_{N,-i} \eqdef {(\pi_\Y)}_\sharp\gammab_{N,-i},
    \]
    so that, evaluating the expectations using the definition of the expectation measure, we obtain 
    \begin{multline*}
        \int_{\X\times\Y} (c + L) \dd \bar \gamma 
        + 
        2
        \int_{\Y\times\Y} H \dd \bar\gamma \otimes \mathbb{E}\gammab_{N,-i} \\
        \ge 
        \mathbb{E}
        \left[
            \int_\Y c(X_i, y) + L(y) \dd \nu_{i,N}^{X_i}
        \right]
        +
        2 
        \int_{\Y\times\Y} H \dd \mathbb{E}[\nu_{i,N}^{X_i}] \otimes \mathbb{E}\gammab_{N,-i}. 
    \end{multline*}

    Rewriting $\gammab_{N,-i} = \gammab_N - \frac{1}{N}\delta_{X_i}\otimes\nub_{i,N}^{X_i}$ and averaging over all $i$, we get that
    \begin{multline*}
        \int_{\X\times\Y} (c + L) \dd \bar \gamma 
        + 
        2\left(1 - \frac{1}{N}\right)
        \int_{\Y\times\Y} H \dd \bar\gamma \otimes \mathbb{E}\gammab_{N}\\ 
        \ge \int_{\X\times\Y} (c + L) \dd \mathbb{E}\gammab_{N}
        + 
        2\int_{\Y\times\Y} H \dd \mathbb{E}\gammab_N\otimes\mathbb{E}\gammab_N
        -
        \frac{2}{N^2} \sum_{i = 1}^N \int_{\Y\times\Y} H \dd \mathbb{E}\nub_{i,N}^{X_i}\otimes \mathbb{E}\nub_{i,N}^{X_i}.
    \end{multline*}

    As $H \in \mathscr{C}_b$, the last term is a $O(1/N)$ and hence vanishes as $N \to \infty$. 
    In addition, since $\mathbb{E}\gammab_N \cvweak{N \to \infty} \gamma$, from the convergence of $\gammab_N$ and Lemma~\ref{lemma:convergente_random_to_deterministic}, we get that 
    \begin{align*}
        0 &\ge 
        \liminf_{N \to \infty}
        \int_{\X\times\Y} c + L \dd (\mathbb{E}\gammab_N - \bar\gamma)
        + 
        \frac{2(N-1)}{N}
        \int_{\Y\times\Y} H \dd \mathbb{E}\gammab_N \otimes (\mathbb{E}\gammab_N - \bar \gamma) \\ 
        &\ge 
        \liminf_{N \to \infty}
        \int_{\X\times\Y} c + L \dd (\gamma - \bar\gamma)
        + 
        2
        \int_{\Y\times\Y} H \dd \gamma\otimes (\gamma - \bar \gamma)
        =
        \inner{\frac{\delta \mathcal{J}}{\delta \gamma}, \gamma - \bar \gamma}.
    \end{align*}
    From Theorem~\ref{thm:potential_structure_cournot_nash}, $\gamma$ is a Cournot--Nash equilibrium.
\end{proof}

\subsection{Convergence in the closed-loop formulation}\label{sec:conv_closed_loop}

\subsubsection{Potential structure in closed loop}
We recall that, in the closed-loop formulation described in the introduction, each player has the knowledge of their own type as well as pf the remaining players'. As a result, given an event $\omega \in \Omega$ and the corresponding realization of the sample ${\left(x_i = X_i(\omega) \right)}_{i \in \mathbb{N}}$, the cost associated with player $i \in \{1,\dots, N\}$ is given by 
\begin{equation*}\label{eq:game_cl_mix}
    \tag{$NP_\omega$}
    \begin{aligned}
        J_{\omega,i}(\nu_i , \nu_{-i}) 
        &\eqdef
        \int_\Y c(x_i, y)\dd \nu_i
        +
        \int_\Y L \dd \nu_i 
        + 
        \frac{2}{N}\sum_{j \neq i}\int_{\Y \times \Y}H \dd \nu_i\otimes \nu_j \\ 
        &= 
        \int_\Y c(x_i, y)\dd \nu_i
        +
        \mathcal{L}(\nu_i) 
        + 
        \frac{2}{N}\sum_{j \neq i}\mathcal{H}(\nu_i, \nu_j).  
    \end{aligned}
\end{equation*}
Notice that we have written the relaxed formulation in mixed strategies, and a profile in pure strategies is just a tuple ${\left(\nu_i\right)}_{i = 1}^N$ such that $\nu_i = \delta_{y_i}$ for all players.

Under assumption~\ref{Hypo-atomless} that $\mu$ has no atoms, which is not restrictive thanks to Remark~\ref{remark.atomless_non_restrictive}, with full $\mathbb{P}$-probability, the event $\bigcup_{i\neq j}
\left\{
    X_i \neq X_j
\right\}$ has $\mathbb{P}$-probability $1$. Hence, for any $\omega$ in such event, setting $x_i = X_i(\omega)$, there is a bijection between the strategy profiles ${\left(\nu_i\right)}_{i = 1}^{N}$ and the measures $\gamma_N \in \mathscr{P}_{\mu_N(\omega)}(\X\times\Y)$ by means of the disintegration theorem, which guarantees that each such measure is uniquely written as
\begin{equation}\label{eq:disintegration_representation}
    \gamma_N 
    = 
    \frac{1}{N} \sum_{i = 1}^N \delta_{x_i}\otimes\nu_i
\end{equation}
This representation can be seen as a lift of a profile of strategies ${\left(\nu_i\right)}_{i = 1}^N$ to the space of plans $\mathscr{P}(\X \times \Y)$. We can define a potential function in the lifted space as
\begin{equation}\label{eq:potential_function_closedloop}
    \mathcal{J}_{\omega, N}(\gamma_N)
    \eqdef
    \begin{dcases}
        \begin{aligned}
            \int_{\X\times\Y} c\dd \gamma_N
            +
            \frac{1}{N}\sum_{i = 1}^N 
            \mathcal{L}\left(\nu_i\right)\\ {} + 
            \frac{1}{N^2}\sum_{j \neq i}
            \mathcal{H}\left(\nu_i, \nu_j\right),
        \end{aligned}
        &
        \text{if } \gamma_N \in\mathscr{P}_{\mu_N(\omega)}(\X\times\Y),\\ 
        +\infty,& \text{ otherwise},
    \end{dcases} 
\end{equation}
where ${\left(\nu_i\right)}_{i = 1}^N$ denotes the unique profile obtained though the representation~\eqref{eq:disintegration_representation}.

The formulation in pure strategies can then be obtained by considering the potential functional
\begin{equation}\label{eq:potential_function_closedloop_pure}
    J_{\omega,N}(y_1,\dots,y_N)
    \eqdef 
    \mathcal{J}_{\omega,N}
    \left(
        \frac{1}{N}\sum_{i = 1}^N \delta_{(x_i,y_i)}
    \right).
\end{equation}
This is equivalent to restricting $\mathcal{J}_{\omega,N}$ to the set
\[
    \mathscr{P}_{\mu_N}^{\text{pure}}(\X\times \Y)
    \eqdef 
    \left\{
        \gamma_N \in 
        \mathscr{P}_{\mu_N}(\X\times \Y) :
        \gamma_N = \frac{1}{N}\sum_{i = 1}^N\delta_{(x_i,y_i)}
    \right\}.  
\]

As in the open-loop case, the above potential functionals $J_{\omega,N}$ and $\mathcal{J}_{\omega,N}$  admit minimizers since $c$, $L$ and $H$ are l.s.c.\ and $L$ has compact sub-level sets. We shall prove that minimizers for each potential functional yield Nash equilibria for the corresponding game and, in the case that $H$ vanishes in the diagonal and is strictly positive elsewhere, we can prove that any minimizer induces a Nash equilibrium in pure strategies. 
\begin{proposition}\label{proposition:NP-potential_closed_loop}
    The following assertions hold:
    \begin{enumerate}
        \item[(i)] For $\mathbb P$-almost every $\omega \in \Omega$, it is equivalent to minimize $J_{\omega, N}$ and $\mathcal{J}_{\omega, N}$, minimizers of the latter are supported on the set of minimizers of the former, and it holds that
        \begin{equation}\label{eq:equivalences_min_potenfunctions}
            \min_{\Y^{\otimes N}} J_{\omega, N} 
            = 
            \min_{\mathscr{P}_{\mu_N}^{\text{pure}}(\X\times \Y)} \mathcal{J}_{\omega, N}
            =
            \min_{\mathscr{P}_{\mu_N}(\X\times \Y)} \mathcal{J}_{\omega, N}.   
        \end{equation}
        \item[(ii)] Let 
        \[
            \begin{aligned}
                {(y_i)}_{i = 1}^N \in \argmin J_{\omega,N}&, \quad 
                \gamma_N = \frac{1}{N}\sum_{i = 1}^N \delta_{x_i}\otimes\nu_{\omega,i} \in \argmin \mathcal{J}_{\omega,N},
            \end{aligned}
        \]
        then ${(y_i)}_{i = 1}^N$ and ${\left(\nu_{\omega,i}\right)}_{i= 1}^N$ induce Nash equilibria for the game~\eqref{eq:game_cl_mix}.
    \end{enumerate}
\end{proposition}
\begin{proof}
    The first equality in~\eqref{eq:equivalences_min_potenfunctions} comes from the bijection between the set of pure equilibrium measures and $\Y^{\otimes N}$.    
    The second is a direct consequence of the fact that the measures $\gamma_N$ in the domain of $\mathcal{J}_{\omega,N}$ can be written as
    \[
        \gamma_N = \frac{1}{N}\sum_{i = 1}^N\delta_{x_i}\otimes\nu_i,  
    \]
    so that we can write
    \[
        \mathcal{J}_{\omega, N}(\gamma_N) = \int_{\Y^{\otimes N}} J_{\omega, N}(y_1, \dots, y_N)\dd  \nu_1\otimes \dots\otimes \nu_N. 
    \]
    Then, for any admissible $\gamma_N$, we have 
    \[
        \min_{\mathscr{P}_{\mu_N(\omega)}(\X\times\Y)} 
        \mathcal{J}_{\omega,N}(\gamma) \ge \min_{\Y^{\otimes N}} J_{\omega,N}.     
    \]
    Taking $\gamma_N$ with second marginal supported on the set of minimizers of $J_{\omega,N}$ gives the result. The proof of assertion (ii) is analogous to the proof of Proposition~\ref{proposition:NP-potential_open}.
\end{proof}

\begin{remark}
     If $H$ vanishes only in the diagonal and we allow for self-interactions in our game, \emph{i.e.}\ we replace $\mathcal{L}$ with $\mathcal{L}_H(\gamma) = \mathcal{L}(\gamma) + \mathcal{H}(\gamma,\gamma)$, then minimizers of $\mathcal{J}_{\omega, N}$ are of the form 
         \begin{equation}\label{eq:minimizers_atomic}
             \gamma_N = \frac{1}{N}\sum_{i = 1}^N\delta_{(x_i,y_i)}, \text{ where } (y_1,\dots, y_N) \in \argmin J_{\omega, N}.
         \end{equation}
    Indeed, since $H \ge 0$, we have that
    \begin{align*}
        \inf \mathcal{J}_{\omega, N}
        \ge 
        \inf J_{\omega, N} + 
        \frac{1}{N^2}\sum_{i = 1}^{N} \mathcal{H}(\nu_{\omega,i},\nu_{\omega,i})
        \ge  
        \inf \mathcal{J}_{\omega, N},
    \end{align*}
    which means that $\mathcal{H}(\nu_{\omega,i},\nu_{\omega,i}) = 0$ for all $i = 1,\dots, N$. Since $H$ only vanishes in the diagonal, it must hold that $\nu_{\omega,i} = \delta_{y_i}$, hence any minimizer of $\mathcal{J}$ is of the form of~\eqref{eq:minimizers_atomic}.

    With a dual reasoning, if $H = +\infty$ in the diagonal, any minimizer of $\mathcal{J}_{\omega, N}$ is atomless. Indeed, assume that $\gamma_N$ has an atom, \emph{i.e.}, there is a point where $\gamma_N(\{x_i,y_i\}) > 0$, and $H$ explodes in the diagonal, the self interaction term gives $\mathcal{J}_{\omega, N}(\gamma_N) = +\infty$ and it cannot be a minimizer. 
\end{remark}

\subsubsection{\texorpdfstring{$\Gamma$}{Gamma}-convergence for the closed-loop formulation}

Now we move on to the question of the convergence of a sequence of Nash equilibria for the games in closed loop~\eqref{eq:game_cl_mix}. In this case, for each event $\omega$ we have a sequence of games, therefore our result consists of a $\Gamma$-convergence with full $\mathbb{P}$-probability of the sequence of potential functionals $\left(\mathcal{J}_{\omega,N}\right)_{N \in \mathbb N}$. We cannot expect a better result, for instance a $\Gamma$-convergence for all events $\omega$, for example in the event of null probability $\left\{
    X_i = x \text{ for all } i \in \mathbb{N}
\right\}$ convergence does not follow.  

We start by showing a general lemma that gives a weaker criterion for $\Gamma$-convergence with full probability. 
\begin{lemma}\label{lemma.Gamma_conv_prob1}
    Let $(\Omega, \mathcal{F}, \mathbb{P})$ be a probability space, ${\left(\mathscr{F}_{\omega,N}\right)}_{\substack{
        N \in \mathbb{N}\\ 
        \omega \in \Omega
    }}$ be a family of functionals over a Polish space $\X$, and $\mathscr{F}$ be a functional over $\X$ such that
    \begin{enumerate}
        \item there is a set $\Omega_0$ with full $\mathbb{P}$-probability such that, for any $x_N \to x$, the $\Gamma$-$\liminf$ inequality for $\mathscr{F}_{\omega,N}$ holds, \emph{i.e.},
        \[
            \mathscr{F}(x) \le 
            \liminf_{N \to \infty} 
            \mathscr{F}_{\omega,N}(x_N), \text{ for all } \omega \in \Omega_0,
        \]
        \item for each $x\in \X$ there is a set $\Omega_x$ with full $\mathbb{P}$-probability for which we can construct recovery sequences of $\mathscr{F}_{\omega,N}$; in other words, for each $\omega \in \Omega_x$, we can construct a sequence ${\left(x_N\right)}_{N \in \mathbb{N}}$ such that $x_n \cvstrong{N \to \infty }{} x$ and
        \[ 
            \limsup_{N \to \infty} 
            \mathscr{F}_{\omega,N}(x_N)
            \le  \mathscr{F}(x).
        \]
    \end{enumerate}
    Then there is a set $\bar\Omega_0$ with full $\mathbb{P}$-probability such that, for any $\omega \in \bar\Omega_0$, the sequence $\mathscr{F}_{\omega,N}$ $\Gamma$-converges to $\mathscr{F}$.
\end{lemma}
\begin{proof}
    First we claim that there exists a countable and dense set $\mathscr{D} \subset \X$ which is dense in the energy $\mathscr{F}$, \emph{i.e.}, for each $x \in \X$ there is ${(x_N)}_{N \in \mathbb{N}} \subset \mathscr{D}$ such that 
    \begin{equation}\label{eq.dense_in_energy}
        x_N \cvstrong{N \to \infty}{} x 
        \text{ and } 
        \mathscr{F}(x_N) \cvstrong{N \to \infty}{} \mathscr{F}(x). 
    \end{equation}
    See for instance~\cite[Lemma~11.12]{ambrosio2021lectures} for a constructive argument, a simple proof comes from the fact that $\mathbb{R}\times \dom \mathscr{F}$ is separable as an (arbitrary) subset of the separable space $\mathbb{R}\times \X$, since subsets of second countable spaces are second countable.  

    Hence we can define the set $\bar \Omega_0$ as 
    \[
        \bar \Omega_0 \eqdef 
        \Omega_0 \cap 
        \bigcap_{x \in \mathscr{D}} \Omega_x,
    \]
    where $\Omega_0$ denotes the set where the $\Gamma$-$\liminf$ holds for all points $x \in \X$ and $\Omega_x$ denotes the event in which we can construct recovery sequences for $x$. Since $\mathscr{D}$ is countable, it holds that $\mathbb{P}(\bar \Omega_0) = 1$. 

    To prove the $\Gamma$-convergence for each $\omega \in \bar \Omega_0$, we recall the notions of lower and upper $\Gamma$ limits from Section~\ref{subsec:gamma_convergence}, and to conclude it suffices to prove for all $\omega \in \bar \Omega_0$ that $\Gamma\text{-}\liminf \mathscr{F}_{\omega,N} = \Gamma\text{-}\limsup \mathscr{F}_{\omega,N} = \mathscr{F}$. Indeed, item $(1)$ shows that 
    \[
        \mathscr{F} \le \Gamma\text{-}\liminf \mathscr{F}_{\omega,N}, 
        \text{ for all }
        \omega \in \bar \Omega_0. 
    \]
    On the other hand, from item $(2)$, it follows for any $x \in \mathscr{D}$ that
    \[
        \Gamma\text{-}\limsup \mathscr{F}_{\omega,N}(x)
        \le 
        \mathscr{F}(x), 
        \text{ for all }
        \omega \in \bar \Omega_0. 
    \]
    Hence, for any $\omega \in \bar \Omega_0$ and an arbitrarily $x\in \X$, let ${\left(x_N\right)}_{N \in \mathbb{N}}$ be a sequence in $\mathscr{D}$ satisfying~\eqref{eq.dense_in_energy}, so that using the lower semi-continuity of the $\Gamma$-upper limit we have that
    \begin{align*}
        \Gamma\text{-}\limsup_{N \to \infty} \mathscr{F}_{\omega, N}(x)
        \le 
        \liminf_{n \to \infty}
        \left(
            \Gamma\text{-}\limsup_{N \to \infty} \mathscr{F}_{\omega, N}(x_N)
        \right)
        \le 
        \liminf_{n \to \infty}
        \mathscr{F}(x_N)
        =
        \mathscr{F}(x)
    \end{align*}
    which gives the $\Gamma$-convergence with full probability. 
\end{proof}

To apply this lemma, we know from the Glivenko--Cantelli law of large numbers that empirical measures converge $\mathbb{P}$-almost surely. Hence, we consider the set
\begin{equation}\label{eq:set_full_probability}
    \Omega_0
    \eqdef 
    \left\{
        \omega \in \Omega:
        \substack{
            \displaystyle
            X_i(\omega) \neq X_j(\omega), \text{ for } i \neq j\\ 
            \displaystyle
            \mu_N(\omega) \cvweak{N \to \infty} \mu
        }
    \right\}.
\end{equation}
The first condition above is so that the sequence of functionals $\mathcal{J}_{\omega,N}$ is well-defined for any $\omega \in \Omega_0$. From the fact that $\mu$ is atomless and the above discussion, $\mathbb{P}(\Omega_0) = 1$. 

While the $\Gamma$-liminf argument will be similar to the open-loop information structure, for the $\Gamma$-limsup we will use a construction depending on a sequence of random variations of the form
\begin{equation}
    \begin{aligned}
        &\frac{1}{N}\sum_{i = 1}^N L_i
        + 
        \frac{1}{N^2}\sum_{i\neq j} H_{i,j}\\ 
        \text{ where }
        L_i \eqdef  c(&X_i,Y_i) + L(Y_i), 
        \quad 
        H_{i,j} \eqdef   H(Y_i,Y_j),
    \end{aligned}
\end{equation}
where $(X_i,Y_i)\sim \gamma$. The first sum is fortunately the mean of an i.i.d.~sequence, so that from the law of large numbers it must converge to its mean. On the other hand, this is not the case for ${(H_{i,j})}_{i\neq j}$. The following proposition, whose proof is a synthesis of the ideas from~\cite[Chapter~12]{klenke2013probability}, states that such families of random variables also enjoy a law of large numbers.

\begin{proposition}\label{prop:exchangeable_lln}
    Let ${\left(H_{i,j}\right)}_{i \neq j\in \mathbb{N}}$ be a sequence of random variables obtained as
    \[
       H_{i,j} = {\left(\Phi(X_i, X_j)\right)}_{i\neq j \in \mathbb{N}},    
    \]
    where $\Phi\colon\X \times \X \to \mathbb{R}$ is a symmetric function and ${\left(X_i\right)}_{i \in \mathbb{N}}$ is an i.i.d.~sample. Then
    \[
        \frac{1}{N^2}\sum_{\substack{
            i,j = 1\\
            i \neq j
        }}^N  H_{i,j} \xrightarrow[N \to \infty]{} \mathbb{E}[H_{1,2}], \text{ with probability $1$.}    
    \]
\end{proposition}

For the sake of readability of the main ideas employed to prove the $\Gamma$-convergence result, we postpone the proof of the previous proposition to Appendix~\ref{sec:appendix}. Gathering these elements, we conclude with the following $\Gamma$-convergence result. We provide two constructions of recovery sequences. The first one which is analytical and directly provides a recovery sequence for a sequence of games in mixed strategies. From the second one, probabilistic in nature, we obtain a sequence of strategy profiles in pure strategies. This little detail means that the passage to the $\Gamma$-limit in pure strategies automatically performs the relaxation to mixed strategies. 

\begin{theorem}\label{theorem:closed_loop_gamma_conv}
    With full $\mathbb{P}$-probability, the sequences of functionals ${\left(\mathcal{J}_{\omega,N}\right)}_{n \in \mathbb{N}}$ and ${\left(J_{\omega,N}\right)}_{n \in \mathbb{N}}$, defined in~\eqref{eq:potential_function_closedloop} and~\eqref{eq:potential_function_closedloop_pure}, converge to $\mathcal{J}$ in the sense of $\Gamma$-convergence for the narrow topology of $\mathscr{P}(\X\times\Y)$. 
\end{theorem}
\begin{proof}
    It suffices to verify the hypotheses of Lemma~\ref{lemma.Gamma_conv_prob1}. To prove item (1), consider $\omega \in \Omega_0$ defined above in~\eqref{eq:set_full_probability}, and let ${\left(\gamma_N\right)}_{N \in \mathbb{N}}$ be a sequence such that $\gamma_N \in \mathscr{P}_{\mu_N(\omega)}(\X\times\Y)$ and converging to $\gamma$. So we can assume that $\gamma_N$ can be written as $\gamma_N = \displaystyle \frac{1}{N}\sum_{i=1}^N \delta_{x_i}\otimes\nu_{i,N}$, where $\nu_{i,N} \in \mathscr{P}(\Y)$, and set 
    \[
        \nu_N 
        \eqdef 
        {(\pi_\Y)}_\sharp \gamma_N 
        = 
        \frac{1}{N}\sum_{i = 1}^N \nu_{i,N}.
    \]
    For any $\omega \in \Omega_0$, it follows that $\gamma \in \mathscr{P}_\mu(\X\times\Y)$.

    Given $M>0$, define $\displaystyle \mathcal{H}^M(\nu,\nu) \eqdef \int_{\Y\times\Y} H\wedge M \dd \nu\otimes \nu$, so that
    \begin{align*}
        \mathcal{J}_{\omega,N}(\gamma_N) 
        &\ge 
        \int_{\X\times \Y}c \dd \gamma_N + 
        \underbrace{
            \frac{1}{N}
            \sum_{i=1}^N 
            \mathcal{L}(\nu_{i,N})
        }_{= \mathcal{L}(\nu_N)} \\
        & \quad \quad \quad \quad \quad 
        + 
        \underbrace{
            \frac{1}{N^2}
            \sum_{i,j} 
            \mathcal{H}^M(\nu_{i,N}, \nu_{j,N})
        }_{= \mathcal{H}^M(\nu_N,\nu_N) }
        - 
        \frac{1}{N^2}\sum_{i = 1}^N\mathcal{H}^M(\nu_{i,N},\nu_{i,N})
        \\ 
        &=
        \int_{\X\times \Y}c \dd \gamma_N + \mathcal{L}(\nu_N) + \mathcal{H}^M(\nu_N,\nu_N) 
        - 
        \frac{1}{N^2}\sum_{i = 1}^N\mathcal{H}^M(\nu_{i,N},\nu_{i,N})\\ 
        &\ge
        \int_{\X\times \Y}c \dd \gamma_N + \mathcal{L}(\nu_N) + \mathcal{H}^M(\nu_N,\nu_N) 
        - 
        \frac{M}{N}. 
    \end{align*} 
    As a result, since, for each $M>0$, the term $M/N$ on the RHS vanishes as $N \to \infty$, combined with the lower semi-continuity of the remaining terms w.r.t.~narrow convergence, as in Lemma~\ref{lemma.lsc_narrow}, we obtain that
    \begin{align*}
        \liminf_{N \to \infty}  
        \mathcal{J}_{\omega,N}(\gamma_N) 
        &\ge
        \liminf_{N \to \infty}  
        \int_{\X\times \Y}c \dd \gamma_N + \mathcal{L}(\nu_N) + \mathcal{H}^M(\nu_N,\nu_N)\\ 
        &\ge
        \int_{\X\times\Y}c\dd \gamma + \mathcal{L}(\nu) + \mathcal{H}^M(\nu,\nu).
    \end{align*}
    From the monotone convergence theorem,  $\mathcal{H}(\nu, \nu) = \displaystyle \sup_{M>0}\mathcal{H}^M(\nu, \nu)$, and the $\Gamma$-$\liminf$ follows. 

    To verify property (2) from Lemma~\ref{lemma.Gamma_conv_prob1}, fix some $\gamma \in \mathscr{P}_\mu(\X\times\Y)$ and let $\nu = {\left(\pi_\Y\right)}_\sharp\gamma$. As mentionned above, we provide two constructions for recovery sequences.
    
    \smallskip
    
    \noindent\uline{\emph{Analytical construction in mixed strategies.}} By an application of the disintegration theorem we write $\gamma = \nu^x\otimes \mu$ for some Borel family ${(\nu^x)}_{x \in \X}$, \emph{i.e.},
    \[
        \int_{\X\times \Y} \varphi(x,y)\dd\gamma
        =
        \int_\X\left(\int_\Y \varphi(x,y)\dd \nu^x(y)\right)\dd \mu(x), 
        \text{ for all }\varphi \in \mathscr{C}_b(\X\times \Y).      
    \]
    This disintegration family is only $\mu$-a.e.~uniquely defined, but we can fix one such family and, for each $\omega \in \Omega$, define a new transportation plan as $\gamma_N\eqdef \mu_N(\omega)\otimes\nu^x$. Since we have fixed one disintegration family, $\gamma_N \in \mathscr{P}_{\mu_N(\omega)}(\X\times\Y)$ is well-defined for every event $\omega = {(x_i)}_{i \in \mathbb{N}}$. From the definition, it then holds that
    \[
        \int_{\X \times \Y} \varphi(x,y)\dd \gamma_N 
        \eqdef 
        \frac{1}{N}\sum_{i=1}^N \int \varphi(x_i, y)\dd \nu^{x_i}, \text{ for all $\varphi \in \mathscr{C}_b(\X\times \Y)$.}
    \]
    Hence $\gamma_N \in \Pi(\mu_N, \nu_N)$ where 
    $\displaystyle
        \mu_N = \frac{1}{N}\sum_{i=1}^N\delta_{x_i}, \ 
        \nu_N = \frac{1}{N}\sum_{i=1}^N\nu^{x_i}.
    $
    
    Let us prove that $\gamma_{N}$ converges narrowly to $\gamma$ will full probability; indeed from Proposition~\ref{prop:criterion_narrow_conv_countable} we know there is a countable set $\mathcal{K} \subset \mathscr{C}_b(X)$, such that, to prove narrow convergence, it suffices to verify that
    \[
        \int_{\X\times \Y} f(x,y) \dd\gamma_{N} 
        \xrightarrow[N \to \infty]{} 
        \int_{\X\times \Y} f(x,y) \dd\gamma, 
        \text{ for all } f\in \mathcal{K}. 
    \]
    For each $f \in \mathcal{K}$, we compute 
    \[
        \int_{\X\times \Y} f(x,y) \dd\gamma_{N} = \frac{1}{N}\sum_{i = 1}^N\int_{\X\times \Y} f(x_i,y) d \nu^{x_i}(y).    
    \]
    Hence, each term of the sum on the right is a realization of the i.i.d.~sequence of random variables $\displaystyle F_i \eqdef \int_\Y f(X_i,\cdot) \dd \nu^{X_i}$. From the strong law of large numbers, it holds with probability 1 that
    \[
        \int_{\X\times \Y} f \dd\gamma_{N} = \frac{1}{N}\sum_{i = 1}^N F_i(\omega) \xrightarrow[N \to \infty]{}
        \mathbb{E}\left[F_1\right] = \int_X f \dd \gamma. 
    \]
    Let $\Omega_{\gamma,f}$ denote the set of probability 1, which depends on $\gamma$ and $f$, where the above converge holds. Then, defining 
    \[
        \tilde \Omega_\gamma = \Omega_0\cap \bigcap_{f \in \mathcal{K}} \Omega_{\gamma,f}, 
    \]
    we have that $\mathbb{P}(\tilde \Omega_\gamma) = 1$ and, for any $\omega \in \tilde \Omega_\gamma$ it holds that $\gamma_{N} \xrightharpoonup[N \to \infty]{} \gamma$. 

    We now apply a similar argument to the convergence of the energies. Indeed, writing 
    \[
        \mathcal{J}_{\omega,N}(\gamma_N) 
        = 
        \frac{1}{N}\sum_{i = 1}^N 
        \int_\Y
        \left(
            c(x_i, y) + L(y)
        \right)\dd \nu^{x_i}(y) 
        + 
        \frac{1}{N^2} 
        \sum_{j\neq i} 
        \int_{\Y\times \Y} H\dd\nu^{x_i}\otimes\nu_{j,N}.
    \]
    We see that the first sum is the empirical average of the i.i.d.~sequence of random variables $\displaystyle L_i \eqdef \int_\Y (c(X_i, y) + L(y))\dd \nu^{X_i}(y)$ while the double sum can be written in terms of the sequence $H_{i,j} \displaystyle \eqdef \int_{\Y\times\Y} H \dd \nu^{X_i}\otimes \nu^{X_j}$. As a consequence, applying once again the strong law of large numbers, there is a set $\Omega_{L,\gamma}$ with probability $1$, such that for any $\omega \in \Omega_{L,\gamma}$ it holds that
    \begin{align*}
        \frac{1}{N}\sum_{i = 1}^N L_i(\omega)
        \xrightarrow[N \to \infty]{} 
        \mathbb{E}[L_1]    
        &=
        \int_\X 
        \left[
            \int_\Y (c(x, y) + L(y))\dd \nu^{x}(y)     
        \right]\dd \mu(x)\\ 
        &= 
        \int_{\X\times \Y}(c(x,y) + L(y))\dd \gamma 
        = 
        \int_{\X\times\Y} c\dd \gamma + \mathcal{L}(\nu).
    \end{align*}

    For the second term, the random variables ${(H_{i,j})}_{i \neq j}$ are no longer i.i.d., but satisfy the hypothesis of Theorem~\ref{prop:exchangeable_lln} with $\Phi$ given by 
    \[
        \Phi(x_1, x_2) \eqdef \int_{\Y\times\Y} H\dd \nu^{x_1}\otimes\nu^{x_2},    
    \]
    which is symmetric and measurable from the measurability of the family ${(\nu^x)}_{x \in \X}$. We conclude that there is another set $\Omega_{H,\gamma}$ with probability $1$ such that for all $\omega \in \Omega_{H,\gamma}$ it holds that 
    \[
        \frac{1}{N^2} 
        \sum_{j\neq i} 
        \int_{\Y\times \Y} H\dd\nu^{x_i}\otimes\nu_{j,N}
        \xrightarrow[N \to \infty]{}
        \mathbb{E}[H_{1,2}] 
        =
        \mathcal{H}(\nu,\nu).
    \]

    Finally, the set $\Omega_\gamma \eqdef \tilde \Omega_\gamma \cap \Omega_{L,\gamma} \cap \Omega_{H,\gamma}$ has probability $1$ and satisfies item (2). From the thesis of Lemma~\ref{lemma.Gamma_conv_prob1}, the $\Gamma$-convergence with full $\mathbb{P}$-probability follows.
   
\smallskip

\noindent\uline{\emph{Probabilistic construction in pure strategies.}} Recalling our original sample $\mathbf{X} \eqdef {(X_i)}_{i \in \mathbb{N}}$, we claim that there exists an i.i.d.\ sequence ${\left(Y_i\right)}_{i \in \mathbb{N}}$ such that
\[
    {(X_i,Y_i)}_{i \in \mathbb{N}} \text{ is i.i.d.\ with law } \gamma. 
\]
Doing as in the proof of Theorem~\ref{theorem:gamma_conv_openloop}, we define the new probability space 
\[
    (\Omega', \mathcal{F}', \mathbb{P}'), \text{ where } 
    \Omega' \eqdef {\X}^{\otimes \mathbb{N}}, \ 
    \mathbb{P}' \eqdef {\mu}^{\otimes \mathbb{N}}.
\]
Since by construction $\mathbb{P}'$ is non-atomic, for instance from~\cite[Theorem~2]{oxtoby1970homeomorphic},
there exists a map $\mathbf{T}\colon \mathbf{x} \in {\X}^{\otimes \mathbb{N}} \mapsto {\left( S_{i}(\mathbf{x}), T_{i}(\mathbf{x}) \right)}_{i \in \mathbb{N}}$ performing the push-forward 
\[
    {\mathbf{T}}_\sharp \mathbb{P}' = {\gamma}^{\otimes \mathbb{N}}, 
\]
in such a way that the sequence $\mathbf{T}(\mathbf{X})$ is i.i.d.\ with law $\gamma$. It follows that $S_i$ must be given by the projection onto the index $i$ and hence $\mathbf{T}(\mathbf{X}) = {\left(X_i, Y_i\right)}_{i \in \mathbb{N}}$, as we wished.

As a consequence of independence, the Glivenko--Cantelli law of large numbers implies 
\begin{equation}\label{eq:weak_conv_gamma}
    \gamma_N
    \eqdef
    \frac{1}{N}\sum_{i = 1}^N \delta_{(X_i(\omega),Y_i(\omega))}
    \cvweak{N \to \infty} 
    \gamma
\end{equation}
with full probability. To obtain the convergence of the energies, namely
\begin{align}
    \label{eq:conv_energies_gamma}
    \frac{1}{N} \sum_{i=1}^N c(X_i,Y_i) 
    &+ L(Y_i) 
    + 
    \frac{1}{N^2} \sum_{i \neq j} H(Y_i,Y_j) 
    \cvstrong{N \to \infty}{}
    \mathcal{J}(\gamma),
\end{align}
we consider the sequence of real-valued random variables
\[
    \frac{1}{N} \sum_{i=1}^N c(X_i,Y_i) + L(Y_i) 
    + 
    \frac{1}{N^2} \sum_{i \neq j} H(Y_i,Y_j). 
\]
The first sum corresponds to an empirical mean of an i.i.d.\ sequence, so that from the classical law of large numbers one has, with full probability, that
\[
    \frac{1}{N} \sum_{i=1}^N c(X_i,Y_i) + L(Y_i) 
    \cvstrong{N \to \infty}{}
    \mathbb{E}\left[
        c(X_1,Y_1) + L(Y_1)
    \right] 
    = 
    \int_{\X\times\Y} c\dd\gamma 
    + 
    \int_\Y L\dd\nu. 
\]
The second double sum is not i.i.d., but, applying Proposition~\ref{prop:exchangeable_lln}, we obtain that, with full probability, 
\[
    \frac{1}{N^2} \sum_{i \neq j} H(Y_i,Y_j)   
    \cvstrong{N \to \infty}{}
    \mathbb{E}\left[
        H(Y_1,Y_2)
    \right] 
    = 
    \int_{\Y\times\Y} H\dd \nu\otimes\nu. 
\]

Intersecting these sets of full probability, we obtain a set $\Omega_{\gamma}$ depending on the measure $\gamma$ where both~\eqref{eq:weak_conv_gamma} and~\eqref{eq:conv_energies_gamma} hold and $\mathbb{P}(\Omega_{\gamma}) = 1$. From the thesis of Lemma~\ref{lemma.Gamma_conv_prob1}, the $\Gamma$-convergence with full $\mathbb{P}$-probability follows.
\end{proof}

With an analogous proof to the open-loop case, we obtain a result assuring, with full $\mathbb{P}$-probability, the convergence of a particular sequence of Nash equilibria to equilibria of Cournot--Nash type, and whenever $H$ is continuous the convergence of any sequence of Nash equilibria. 

\begin{theorem}\label{theorem:nash2cournotnash_closedloop}
    Under the same assumptions of Theorem~\ref{thm:minimizer_is_equilibria}, with full $\mathbb{P}$-pro\-ba\-bi\-lity, there are sequences of Nash equilibria for the game~\eqref{eq:game_cl_mix}, described by transportation plans ${\left(\gamma_N\right)}_{N \in \mathbb{N}}$ converging in the narrow topology, up to a subsequence, to a Cournot--Nash equilibrium $\gamma\in\mathscr{P}_\mu(\X\times\Y)$ as in Definition~\ref{def:equ_CournotNash}. 
    
    Assuming in addition that $H \in \mathscr{C}_b(\Y\times \Y)$, with $\mathbb{P}$-full probability, for any sequence of Nash equilibria ${\left(\gamma_N\right)}_{N \in \mathbb{N}}$ from game~\eqref{eq:game_cl_mix}, that is 
    \begin{equation*}
        \gamma_N 
        \eqdef 
        \frac{1}{N}\sum_{i = 1}^N \delta_{x_i}\otimes \nu_{i,N} \in 
        \mathscr{P}_{\mu_N(\omega)}(\X\times\Y),
    \end{equation*}
    converging to $\gamma$ in the narrow topology of $\mathscr{P}(\X\times\Y)$, it holds that 
    $\gamma \in \mathscr{P}_\mu(\X\times\Y)$, and it is a Cournot--Nash equilibrium in the sense of Definition~\ref{def:equ_CournotNash}. 
\end{theorem}

\section{Conclusion}\label{sec.conclusion}
If anything, the convergence of Nash to Cournot--Nash equilibria demonstrates how difficult the convergence question in the context of Mean Field Games is. The $\Gamma$-convergence approach relies entirely on the fact that a variational description of equilibria is provided in Theorem~\ref{thm:potential_structure_cournot_nash} and is not useful to non-variational games, for which we have only fixed point techniques at our disposal. The following questions then present themselves:

\begin{itemize}
    \item Can we consider other types of energy? The analysis seems very specific to an energy that is the sum of an individual and a pairwise interaction costs. 
    \begin{itemize}
        \item Using the characterization of convex functions as the envelope of all linear functions below it, one could try to adapt the arguments of the linear term to the case of an individual convex energy. 
        \item In principle, the arguments treating the pairwise interaction term could be extended to a $k$-wise interaction, as long as the number of players interacting remains uniformly bounded, as in this case an analogous law of large numbers from Appendix~\ref{sec:appendix} holds. 
    \end{itemize}
    \item Another direction would be to derive a large deviations principle for the Gibbs measures associated with the potential function of the $N$-player games, whose Nash equilibria converge to Cournot--Nash equilibria, in accordance with the statistical mechanics intuition that motivated the original name Mean Field Games.
    \item The passage to the limit of games in pure strategies was achieved in the closed-loop formulation via the probabilistic construction of recovery sequences. The same strategy does not work for the open-loop case since adapting it would give a profile of strategies induced by random variables of the form
    \[
        Y_i = T_{N,i}(X_i, X_{-i}),
    \]
    being therefore suitable for a stochastic closed-loop formulation, which is interesting in its own right, but fails to capture the situation modelled by the open-loop formulation. On the other hand, concluding the $\Gamma$-$\liminf$ for this alternative formulation seems challenging since we do not have the equivalent of Lemma~\ref{lemma:convergente_random_to_deterministic}.
\end{itemize}

\appendix
\section{The law of large numbers for symmetric functions of an i.i.d.\ sample}\label{sec:appendix}
In this appendix we prove Proposition~\ref{prop:exchangeable_lln}. The ideas are a minor modification of the presentation of~\cite{klenke2013probability}, hence our goal is to make it as self-contained as possible to readers less familiarized with probability theory, but we hope it can be useful in other contexts as well. We also observe that this proof remains true if one considers $\Phi\colon\X^{\otimes k} \to \mathbb{R}$, for any $k \in \mathbb{N}$. With this we can now proceed with our $\Gamma$-convergence type result.

\begin{proof}[Proof of Proposition~\ref{prop:exchangeable_lln}]
    First, define the exchangeable $\sigma$-algebra as follows: we say a function $f\colon\mathbb{R}^{\otimes\mathbb{N}} \to \mathbb{R}$ is $n$-symmetric if it is symmetric w.r.t.~permutations of indexes smaller than $n$. In other words, for any permutation $\sigma\colon\mathbb{N}\to \mathbb{N}$ such that $\sigma(k) = k$ for $k> n$, then $f\left({(x_{\sigma(n)})}_{n \in \mathbb{N}}\right) = f\left({(x_{n})}_{n \in \mathbb{N}}\right)$. We then define the exchangeable $\sigma$-algebra as
    \[
        \mathcal{E}_{\infty} \eqdef \bigcap_{n \in \mathbb{N}}\mathcal{E}_n, 
        \text{ where }    
        \mathcal{E}_n \eqdef 
        \sigma\left(
            \left\{
                f\left({(X_i)}_{i \in \mathbb{N}}\right) : 
                \begin{aligned}
                    &f\colon \mathbb{R}^{\otimes\mathbb{N}} \to \mathbb{R}\\ 
                     & \text{ is $n$-symmetric and Borel}
                \end{aligned}
            \right\}
        \right), 
    \]
    where $\sigma\left({\left\{F_i\right\}}_{i \in I}\right)$ is defined as the smallest $\sigma$-algebra that makes the whole family of random variables ${(F_i)}_{i\in I}$ measurable. 

    Take $g \colon \mathbb{R}^{\otimes \mathbb{N}} \to \mathbb{R}$ a bounded and $n$-symmetric function. For all $i \le n$, as the sample ${\left(X_i\right)}_{i\in \mathbb{N}}$ is i.i.d., we have
    \begin{align*}
        & \mathbb{E}\left[ H_{i,j} g\left( X_\cdot \right)\right] \\
        {} = {} &
        \mathbb{E}\left[\Phi(X_i, X_j) g\left(X_1, X_2, X_3, \dots, X_{i-1}, X_i, X_{i+1}, \dots, X_{j-1}, X_j, X_{j+1}, \dots\right)\right]\\ 
        {} = {} &
        \mathbb{E}\left[\Phi(X_1, X_2) g\left(X_i, X_j, X_3, \dots, X_{i-1}, X_1, X_{i+1}, \dots, X_{j-1}, X_2, X_{j+1}, \dots\right)\right]\\ 
        {} = {} &
        \mathbb{E}\left[ H_{1,2} g\left( X_\cdot \right)\right]
    \end{align*}
    In particular, taking $g = 1_A$ for an arbitrary set $A \in \mathcal{E}_n$ and averaging the above equality for all $1 \le i,j \le n$ with $i \neq j$, we obtain that
    \[
        \mathbb{E}\left[
            \frac{1}{n(n-1)}\sum_{i \neq j} H_{i,j} 1_A
        \right] 
        = 
        \mathbb{E}\left[
            H_{1,2} 1_A
        \right], 
        \text{ so that }
        \frac{1}{n(n-1)}\sum_{i \neq j} H_{i,j} 
        = 
        \mathbb{E}\left[H_{1,2} \mid \mathcal{E}_n \right],
    \]
    by the definition of conditional expectation for $L^1$ random variables. This result can also be found in~\cite[Theorem~12.10]{klenke2013probability}. This means that $\frac{1}{n(n-1)}\sum_{i \neq j} H_{i,j}$ is a backwards martingale for the filtration ${\left(\mathcal{E}_n\right)}_{n \in \mathbb{N}}$ and a suitable martingale convergence theorem,~\cite[Theorem~12.14]{klenke2013probability}, gives that
    \[
        \frac{1}{n(n-1)}\sum_{i \neq j} H_{i,j}
         \xrightarrow[n \to \infty]{} 
         \mathbb{E}\left[
            H_{1,2} \mid \mathcal{E}_\infty 
        \right] 
        \text{ with convergence a.s.~and in $L^1$.} 
    \]

    Since ${(X_i)}_{i \in \mathbb{N}}$ is i.i.d., the Hewitt--Savage $0-1$ law, see~\cite[Corollary~12.19]{klenke2013probability} and~\cite{hewitt1955symmetric}, states that $\mathcal{E}_\infty$ is a trivial $\sigma$-algebra, so that for any set $A \in\mathcal{E}_\infty$, $\mathbb{P}(A)$ is either $0$ or $1$. 
    Hence, as $\mathbb{E}\left[H_{1,2}\mid \mathcal{E}_\infty \right]$ is an $\mathcal{E}_\infty$-adapted random variable, it must be given by a constant, given by its mean 
    $\mathbb{E}\left[\mathbb{E}\left[H_{1,2} \mid \mathcal{E}_\infty \right]\right] = \mathbb{E}\left[ H_{1,2}\right]$, and the result follows. 
\end{proof}

\let\OLDthebibliography\thebibliography
\renewcommand\thebibliography[1]{
  \OLDthebibliography{#1}
  \setlength{\parskip}{1pt plus 1pt minus 1pt}
  \setlength{\itemsep}{1pt plus 1pt minus 1pt}
}

\bibliographystyle{abbrv}
\bibliography{main}

\end{document}